\documentclass[12pt]{scrartcl}
\usepackage[english]{babel}
\usepackage[utf8x]{inputenc}
\usepackage{amsfonts}
\usepackage{amsmath}
\usepackage{amssymb}
\usepackage{amsthm}
\usepackage{mathrsfs}
\usepackage[colorlinks=false,urlcolor=blue]{hyperref}
\usepackage{listings}
\usepackage{graphicx}
\usepackage{enumerate}
\usepackage{mathtools}
\usepackage{tensor}
\usepackage{tabulary}
\usepackage{multirow}
\usepackage{xcolor}
\definecolor{questype}{gray}{0.8}

\addtokomafont{section}{\rmfamily}
\addtokomafont{sectionentry}{\rmfamily}
\addtokomafont{subsection}{\rmfamily}
\addtokomafont{paragraph}{\rmfamily}

\swapnumbers
\theoremstyle{plain}
\newtheorem{theorem}[paragraph]{Theorem}

\newtheorem{corollary}[paragraph]{Corollary}

\theoremstyle{definition}
\newtheorem{definition}[paragraph]{Definition}
\newtheorem{remark}[paragraph]{Remark}

\counterwithin{paragraph}{section}
\counterwithout{paragraph}{subsection}
\counterwithout{paragraph}{subsubsection}
\counterwithin{equation}{section}

\newcommand\circledmark[2][white]{%
  \ooalign{%
    \hidewidth
    \kern0.7ex\raisebox{-0.7ex}{\scalebox{2.5}{\textcolor{#1}{\textbullet}}}
    \hidewidth\cr
    #2\cr
  }%
}

\def\formtmp#1#2{{\vskip12pt\noindent\fboxsep=0pt\colorbox{#1}{\vbox{\vskip3pt\hbox to \textwidth{\hskip3pt\vbox{\raggedright\noindent\textbf{#2\vphantom{Qy}}}\hfill}\vspace*{3pt}}}\par\vskip2pt%
\noindent\kern0pt}}

\newenvironment{questype}[1]{\ignorespaces\def\stmtopen##1{##1}%
\formtmp{questype}{\ \circledmark[white]{\textcolor{black}{\bfseries?}}{\kern6pt}#1}}{\par\noindent\textcolor{questype}{\rule{\columnwidth}{1pt}}\vskip2pt\par\addvspace{\baselineskip}}%

\newcommand{\R}{\mathbb{R}}
\newcommand{\C}{\mathbb{C}}
\newcommand{\N}{\mathbb{N}}
\newcommand{\Z}{\mathbb{Z}}
\newcommand{\Ham}{\mathbb{H}}
\newcommand{\Oct}{\mathbb{O}}
\newcommand{\Id}{\mathrm{Id}}
\newcommand{\Ric}{\mathrm{Ric}}
\newcommand{\scal}{\mathrm{scal}}
\newcommand{\vol}{\mathrm{vol}}
\newcommand{\Vol}{\mathrm{Vol}}
\newcommand{\Sym}{\operatorname{Sym}}
\newcommand{\tr}{\operatorname{tr}}
\newcommand{\im}{\operatorname{im}}
\newcommand{\diag}{\operatorname{diag}}
\newcommand{\End}{\operatorname{End}}
\newcommand{\Aut}{\operatorname{Aut}}
\newcommand{\Hom}{\operatorname{Hom}}
\newcommand{\Ad}{\operatorname{Ad}}
\newcommand{\Cas}{\mathrm{Cas}}
\newcommand{\Met}{\mathscr{M}}
\newcommand{\Diff}{\mathscr{D}}
\newcommand{\EMod}{\mathscr{E}}
\newcommand{\Sy}{\mathscr{S}}
\newcommand{\X}{\mathfrak{X}}
\newcommand{\U}{\operatorname{U}}
\newcommand{\SU}{\operatorname{SU}}
\newcommand{\su}{\mathfrak{su}}
\newcommand{\SO}{\operatorname{SO}}
\newcommand{\Orth}{\operatorname{O}}
\newcommand{\PSO}{\operatorname{PSO}}
\newcommand{\so}{\mathfrak{so}}
\newcommand{\Sp}{\operatorname{Sp}}
\newcommand{\Spin}{\operatorname{Spin}}
\newcommand{\GL}{\operatorname{GL}}
\newcommand{\gl}{\mathfrak{gl}}
\newcommand{\SL}{\operatorname{SL}}
\newcommand{\rmS}{\mathrm{S}}
\newcommand{\pr}{\operatorname{pr}}
\newcommand{\mf}{\mathfrak{m}}
\newcommand{\hf}{\mathfrak{h}}
\newcommand{\gf}{\mathfrak{g}}
\newcommand{\kf}{\mathfrak{k}}
\DeclareMathOperator*{\closedsum}{\overline{\bigoplus}}
\newcommand{\TT}{\Sy^2_{\mathrm{tt}}}
\newcommand{\ef}{\mathfrak{e}}
\newcommand{\tf}{\mathfrak{t}}
\newcommand{\rk}{\operatorname{rank}}
\newcommand{\Isom}{\operatorname{Isom}}
\newcommand{\Ltwoinprod}[2]{\big(#1,#2\big)_{L^2}}
\newcommand{\intprod}{\mathbin{\lrcorner}}
\newcommand{\FN}[2]{[#1,#2]^{\mathrm{FN}}}
\newcommand{\Proj}{\mathbb{P}}
\newcommand{\rmE}{\mathrm{E}}
\newcommand{\rmF}{\mathrm{F}}
\newcommand{\rmG}{\mathrm{G}}
\newcommand{\LL}{\Delta_{\mathrm{L}}}
\newcommand{\EH}{\mathcal{S}}
\newcommand{\Csc}{\mathscr{C}}
\newcommand{\K}{\mathcal{K}}
\newcommand{\Univ}{\mathfrak{U}}
\newcommand{\Lie}{\mathcal{L}}
\newcommand{\Gr}{\mathrm{Gr}}
\newcommand{\Y}{\mathcal{Y}}
\newcommand{\A}{\mathcal{A}}
\newcommand{\Eop}{\mathbf{E}}
\newcommand{\Slice}{\mathscr{S}}
\newcommand{\Hol}{\mathrm{Hol}}
\newcommand{\Spinor}{\Sigma}
\renewcommand{\mod}{\mathbin{\mathrm{mod}}}
\newcommand{\rmc}{\mathrm{c}}
\newcommand{\coind}{\mathrm{coind}}
\newcommand{\qH}{\mathbb{H}}
\newcommand{\qE}{\mathbb{E}}
\newcommand{\D}{\mathrm{d}}

\title{\rmfamily Einstein metrics, their moduli spaces and stability}
\author{Paul Schwahn\footnote{Universidade Estadual de Campinas, IMECC, Rua Sérgio Buarque de Holanda 651, 13083-859 Campinas-SP, Brazil. Email: \texttt{schwahn@ime.unicamp.br}}, Uwe Semmelmann\footnote{Institut f\"ur Geometrie und Topologie, Fachbereich Mathematik, Universit\"at Stuttgart, Pfaffenwaldring 57, 70569 Stuttgart, Germany. Email: \texttt{uwe.semmelmann@mathematik.uni-stuttgart.de}}}
\date{\today}

\begin{document}

\maketitle

\begin{abstract}
\noindent
This survey deals with two closely connected topics: first, the stability of Einstein metrics under the Einstein--Hilbert functional, and second, their deformation theory and the study of the moduli space of Einstein metrics on a compact manifold. To first order, both problems reduce to studying the spectrum and eigentensors of the Lichnerowicz Laplacian. We give an introduction to the classical theory and survey recent results and advances.

\medskip


\medskip

\end{abstract}

\tableofcontents

\section{Introduction}

In Riemannian and pseudo-Riemannian geometry, a metric $g$ is called \emph{Einstein} if its Ricci tensor is proportional to $g$ itself, that is $\Ric_g=Eg$ for some constant $E\in\R$. Originally introduced in the context of General Relativity, these metrics have attracted considerable interest and are widely studied today also in the Riemannian setting. The Einstein equation $\Ric_g=Eg$ is a nonlinear partial differential equation and notoriously hard to solve. Its solution space on a fixed manifold, modulo gauge symmetries, is called the \emph{Einstein moduli space}. The global structure of this moduli space remains largely mysterious in the general case.

Einstein metrics may also be characterized variationally via the \emph{Einstein--Hilbert action}, a functional on the space of all Riemannian metrics. The \emph{stability} of an Einstein metric $g$ refers to the local behavior of this functional around $g$, and it is closely connected to the local structure of the moduli space around $g$.

This survey aims to give a comprehensive introduction to the study of stability as well as the deformation theory and moduli spaces of Einstein metrics. We will explain the relation between these topics and their connection to adjacent domains of study like the Ricci flow or the Yamabe functional. Further, we give a broad overview over the current state of knowledge in particular geometric situations, such as on manifolds with special holonomy or homogeneous spaces. This includes collecting and stating some open questions in the field that have attracted interest and that we hope will continue to do so.

Throughout this article, we are going to restrict our considerations to \emph{compact} manifolds, mentioning the noncompact world only in passing. The article \cite{Kr25} by Kröncke, appearing in this same volume, discusses stability and related concepts also on noncompact manifolds, and is in some sense orthogonal to this one.

\subsection{History and motivation}

The field equations of gravitation in General Relativity, first formulated by and named after Albert Einstein in 1915, can be written as
\[\Ric_g-\frac{\scal_g}{2}g=T,\]
where $g$ is a Lorentzian metric on a spacetime $M^n$, and $T$ is the energy-momentum tensor. In a vacuum ($T=0$) they are equivalent to the metric being Ricci-flat, $\Ric_g=0$. Two years after formulating the field equations, Einstein  realized that the principles from which he derived his equations allowed adding another ``cosmological'' term, leading to the modified field equations
\[\Ric_g-\frac{\scal_g}{2}g+\Lambda g=T,\]
where $\Lambda\in\R$ is the \emph{cosmological constant}. The vacuum solutions of this equation equivalently satisfy $\Ric_g=Eg$ with $E=\frac{2}{n-2}\Lambda$.

Also in 1915, Hilbert shows that the vacuum Einstein equations (without cosmological term) arise as the Euler--Lagrange equations of the action  $\EH$ given by
\[\EH(g)=\int_M\scal_g\vol_g,\]
today known as the \emph{Einstein--Hilbert action}. Here we need to assume $M$ to be compact, or at least the scalar curvature to have suitable decay for the integral to converge. By restricting $\EH$ to metrics of a given total volume, one obtains as critical points precisely the metrics with $\Ric_g=Eg$.

The interest in Einstein metrics from Riemannian geometry came, beside the potential applications in physics, from the search of a ``best'' metric on a given manifold. An Einstein manifold $(M,g)$ has ``constant Ricci curvature'' in the sense that the quadratic form $X\mapsto\Ric_g(X,X)$, $X\in TM$, is constant on the unit sphere bundle inside $TM$. The constancy of the Ricci curvature seems to be a sweet spot between the stronger condition of constant sectional curvature (which forces $M$ to be a quotient of a round sphere, Euclidean space, or hyperbolic space) and the weaker condition of constant scalar curvature (since the resolution of the Yamabe problem we know that every metric is conformal to a constant scalar curvature metric). Moreover, the Einstein--Hilbert action is the simplest curvature functional one can write down on a general Riemannian manifold, lending credence to the special status of Einstein metrics.

The behavior of Einstein metrics as critical points of the Einstein--Hilbert action may vary. We call an Einstein metric \emph{stable} (or $\EH$-stable) if it is a local maximum of $\EH$ restricted to the set of constant scalar curvature metrics of fixed total volume. Einstein metrics are also (up to rescaling) fixed points of the Ricci flow, and the dynamical behavior of this flow close to an Einstein metric is indeed closely related to its $\EH$-stability. This will shortly be elucidated.

The solution space of the Einstein equation is generally not well understood. We will not touch upon the question of which manifolds admit Einstein metrics, but instead mostly focus on the local structure of the Einstein moduli space around a given Einstein metric $g$. In particular, we shall discuss the question whether $g$ can be deformed to other non-homothetic Einstein metrics. If not, $g$ is called \emph{rigid}.

A first step towards understanding this local behavior is to solve the linearized Einstein equation, which in the right gauge is an eigenvalue problem for an elliptic operator called the \emph{Lichnerowicz Laplacian}, a natural generalization of the Hodge--de Rham Laplacian to tensors of any type. This connects the stability and rigidity questions to the spectral geometry of $(M,g)$.

\subsection{Plan of the article}

We begin in Section~\ref{sec:stable} with reviewing the properties of the Einstein--Hilbert action to first and second order, comparing the different stability notions for Einstein metrics and relating them to other concepts such as scalar curvature rigidity, the Yamabe functional and the Ricci flow. In Section~\ref{sec:moduli} we turn to the Einstein moduli space and the deformation theory of Einstein metrics.

The remainder of the article surveys what we know about the questions of stability and rigidity in particular geometric settings. First, in Section~\ref{sec:prodfib}, we discuss how stability and rigidity behave under taking products, and what we know about warped product metrics and total spaces of Riemannian submersions. In Section~\ref{sec:parallelspinor} we specialize to metrics admitting a parallel spinor, which implies that they are Ricci-flat. The other cases in Berger's holonomy list, that is Kähler and quaternion-Kähler metrics, are discussed in Sections~\ref{sec:ke} and \ref{sec:qk}, respectively. Next, in Section~\ref{sec:killingspinor}, we turn to a particular kind of non-integrable geometries: manifolds admitting Killing spinors. Finally, in Section~\ref{sec:homogeneous}, we discuss the rich world of homogeneous Einstein manifolds, on which the Einstein equation and its linearization can be treated by algebraic and representation-theoretic methods.

\section{The Einstein--Hilbert action and stability}
\label{sec:stable}

Let $M^n$ be a compact manifold of dimension $n\geq3$, and denote with $\Met$ the set of all Riemannian metrics on $M$. This is a convex open cone in $\Sy^2(M)=\Gamma(\Sym^2T^\ast M)$, the space of smooth symmetric $2$-tensor fields.

\begin{definition}
The \emph{Einstein--Hilbert action} is the functional $\EH: \Met\to\R$ defined by
\[\EH(g)=\int_M\scal_g|\vol_g|,\qquad g\in\Met.\]
\end{definition}
Here $|\vol_g|$ denotes the volume density associated to $g$; if $M$ is orientable, one may replace it by the volume form $\vol_g\in\Omega^n(M)$.

\subsection{First and second order properties}

The vector space $\Sy^2(M)$ is ILH (an inverse limit of Hilbert spaces), and in particular a Fréchet space, hence there is a well-behaved notion of differentiation. The \emph{first variation} of $\EH$ at some $g\in\Met$ is given by
\begin{equation}
\EH'_g(h)=\Ltwoinprod{-\Ric_g+\tfrac{\scal_g}{2}g}{h}.
\end{equation}

We may always normalize a metric $g$ to have total volume one without posing a restriction for the following discussion. Let
\[\Met_1=\left\{g\in\Met\,\middle|\,\Vol(g)=1\right\}\]
denote the space of metrics of total volume one. This is an ILH-submanifold of $\Met$ with (formal) tangent space
\[T_g\Met_1=\left\{h\in\Sy^2(M)\,\middle|\,\int_M(\tr_g h)|\vol_g|=0\right\}.\]
The following theorem \cite[Thm.~4.21]{besse} essentially goes back to Hilbert.
\begin{theorem}
\label{thm:hilbertcrit}
For $g\in\Met_1$, the following conditions are equivalent:
\begin{enumerate}[\upshape(i)]
    \item $g$ is Einstein,
    \item $g$ is a critical point of $\EH$ restricted to $\Met_1$,
    \item $g$ is a critical point of the \emph{volume-normalized Einstein--Hilbert action}
    \[\widetilde\EH: \Met\to\R,\qquad\widetilde\EH(g)=\Vol(g)^{\frac{2-n}{n}}\EH(g).\]
\end{enumerate}
\end{theorem}

Let $\Diff$ denote the diffeomorphism group of $M$. The Einstein--Hilbert functional is invariant under diffeomorphisms, that is
\[\EH(g)=\EH(\varphi^\ast g)\]
for all $g\in\Met$, $\varphi\in\Diff$. The tangent space of the $\Diff$-orbit of a metric $g$ is given by
\[T_g(\Diff\cdot g)=\Lie_{\X(M)}g=\{\Lie_Xg\,|\,X\in\X(M)\}.\]
This is the image of the \emph{Killing operator}, a first order, overdetermined elliptic operator
\[\delta^\ast: \Sy^k(M)\to\Sy^{k+1}(M),\qquad\alpha\mapsto\sum_ie^i\odot\nabla_{e_i}\alpha,\]
which is given by $\delta^\ast\alpha=\Lie_{\alpha^\sharp}g$ for $\alpha\in\Sy^1(M)=\Omega^1(M)$. Here $\nabla$ is the Levi-Civita connection associated to $g$. 

\begin{remark}
We warn the reader that there are different conventions for the symmetric and inner product on symmetric tensors, and hence for the Killing operator. Our convention is analogous to the usual wedge and inner product on alternating tensors. The symmetric product is given by
\[\alpha\odot\beta=\frac{(k+l)!}{k!l!}\mathrm{sym}(\alpha\otimes\beta)\]
for any $\alpha\in\Sym^kT$, $\beta\in\Sym^lT$, where $\mathrm{sym}$ denotes the $\GL(T)$-invariant projection $T^{\otimes k}\to\Sym^kT$. For example, $X\odot Y=X\otimes Y+Y\otimes X$ for $X,Y\in T$. The inner product in turn is determined by stipulating that $X\intprod: \Sym^{k+1}T\to\Sym^kT$ is adjoint to $X\odot: \Sym^kT\to\Sym^{k+1}T$ for any $X\in T$.
\end{remark}

The Killing operator is formally adjoint to the \emph{divergence operator}
\[\delta: \Sy^{k+1}(M)\to\Sy^k(M),\qquad \alpha\mapsto-\sum_ie_i\intprod\nabla_{e_i}\alpha.\]
Hence, by elliptic theory, there is an $L^2$-orthogonal decomposition
\begin{equation}
\Sy^2(M)=\Lie_{\X(M)}g\oplus\ker\delta.
\label{diffdecomp}
\end{equation}
The condition $\delta h=0$ is tantamount to fixing a gauge with respect to the symmetry group $\Diff$. It is called the \emph{transverse gauge}.

Since any Einstein metric is a critical point of $\EH|_{\Met_1}$ by Theorem~\ref{thm:hilbertcrit}, we may want to analyze its behavior by looking at the \emph{second variation} $\EH_g''$, a bilinear form on $T_g\Met_1$. 
Recall that the index (coindex) of a bilinear form is the dimension of a maximal subspace on which it is negative (positive) definite, and the nullity is the dimension of its kernel.
The following result is due to Berger and Koiso.

\begin{theorem}
\label{thm:secondvar}
If $(M,g)$ is Einstein and not isometric to the standard sphere, there is a $\EH_g''$-orthogonal decomposition
\begin{equation}
T_g\Met_1=C^\infty_\perp(M)g\oplus\Lie_{\X(M)}g\oplus\TT(M),\label{EHdecomp}
\end{equation}
where
\begin{enumerate}[\upshape(i)]
    \item $C^\infty_\perp(M)$ denotes the space of smooth functions with average zero, and $\EH_g''$ is positive definite on $C^\infty_\perp(M)g$,
    \item $\Lie_{\X(M)}g$ is contained in the null space of $\EH_g''$,
    \item $\TT(M)=\ker\tr_g\cap\ker\delta$ denotes the space of transverse traceless (\emph{tt}) tensors, and $\EH_g''$ has infinite index and finite coindex and nullity on $\TT(M)$.
\end{enumerate}
\end{theorem}

\begin{remark}
The standard sphere $(S^n,g_0)$ is excluded in the preceding theorem because $C^\infty_\perp(S^n)g_0$ and $\Lie_{\X(S^n)}g_0$ intersect nontrivially; the result becomes true if one replaces $C^\infty_\perp(S^n)$ by the orthogonal complement to the first order spherical harmonics, i.e.~the Laplace--Beltrami eigenfunctions to the eigenvalue $\scal_{g_0}/(n-1)$.
\end{remark}

The ``interesting'' part of the decomposition in Theorem~\ref{thm:secondvar} is clearly the (infinite-dimensional) space of tt-tensors, where $\EH_g''$ is ``almost'' negative definite, up to a finite-dimensional space of positive or null directions. The tt-tensors are tangent to the ILH-submanifold
\[\Csc=\{g\in\Met_1\,|\,\scal_g=\text{const.}\}\]
of unit volume metrics of constant scalar curvature; indeed, its tangent space is
\[T_g\Csc=\Lie_{\X(M)}g\oplus\TT(M).\]
This motivates the following definition.

\begin{definition}[Stability]
\label{def:stable}
An Einstein metric $g$ is called
\begin{enumerate}[\upshape(i)]
    \item \emph{strictly (linearly) stable} if $\EH_g''$ is negative-definite on $\TT(M)$,
    \item \emph{$\EH$-stable} if $g$ is a local maximum of $\EH$ restricted to $\Csc$,
    \item \emph{(linearly) semistable} if $\EH_g''$ is negative-semidefinite on $\TT(M)$,
    \item \emph{linearly unstable} if there exists $h\in\TT(M)$ such that $\EH_g''(h,h)>0$,
    \item \emph{$\EH$-unstable} if it is not $\EH$-stable.
\end{enumerate}
\end{definition}
Clearly, we have the implications (i)$\Rightarrow$(ii)$\Rightarrow$(iii) and (iv)$\Rightarrow$(v). We call the coindex of $\EH_g''$ on $\TT(M)$ the \emph{Hilbert coindex} of the Einstein metric $g$ and denote it by $\coind(g)$. 
Tensors as in (iv) are called \emph{destabilizing directions}.

\begin{definition}[Infinitesimal Einstein deformations]
Let $g$ be an Einstein metric. An \emph{infinitesimal Einstein deformation} (IED) of $g$ is a tt-tensor that is a null direction for $\EH_g''$. We denote the space of IED by
\[\varepsilon(g)=\ker\EH_g''|_{\TT(M)}.\]
\end{definition}

One may restrict the Einstein--Hilbert action further and ask about the meaning of critical points. For example, a metric $g\in\Met_1$ has constant scalar curvature precisely if it is a critical point of $\EH$ restricted to $[g]\cap\Met_1$, where $[g]$ denotes the class of metrics conformal to $g$. On the other hand, it is unknown what the critical points of $\EH$ restricted to $\Csc$ are. This is the subject of the famous Besse conjecture \cite[4.48]{besse}.
\begin{questype}{Besse conjecture}
Let $g$ be a critical point of $\EH$ restricted to $\Csc$. Then $g$ is an Einstein metric.
\end{questype}

\subsection{The Yamabe functional}
\label{sec:yamabe}

A given manifold may admit many different Einstein metrics and thus many candidates for the ``best'' metric. Therefore one may look for some sort of global extremality condition that could single out a ``best'' Einstein metric. Due to Einstein metrics always being saddle points of $\EH|_{\Met_1}$, there is no sensible notion of local or even global extrema for $\EH$. In 1960, this led Yamabe to the following min-max approach.

\begin{definition}[Yamabe functional and invariant]
\begin{enumerate}[\upshape(i)]
    \item If $c\subset\Met$ is a conformal class on $M$, then
    \[\Y(M,c)=\inf_{g\in c}\widetilde{\EH}(g)=\inf_{g\in c\cap\Met_1}\EH(g)\]
    is called the \emph{Yamabe constant} of $c$. A metric $g\in c$ achieving the infimum is called a \emph{Yamabe metric}.
    \item The supremum over all conformal classes,
    \[\Y(M)=\sup_{c}\Y(M,c),\]
    is called the \emph{Yamabe invariant} of $M$.
\end{enumerate}
\end{definition}

Trudinger--Aubin--Schoen showed that every conformal class contains a Yamabe metric, that Yamabe metrics always have constant scalar curvature, and that
\[\Y(M)\leq\widetilde{\EH}(S^n,g_0)\]
for any compact manifold $M$. Moreover, Einstein metrics, or metrics of constant nonpositive scalar curvature, are always Yamabe metrics. A good survey on this is \cite{LeB99}. We note one additional fact that we will refer to later:

\begin{theorem}
\label{thm:yamabepos}
$\Y(M)>0$ if and only if $M$ admits a metric of positive scalar curvature.
\end{theorem}

An Einstein metric $g$ which is a global maximum of the Yamabe functional in the sense that $\widetilde{\EH}(g)=\Y(M)$ is called \emph{supreme}. For example, the standard metrics on $S^n$ and $\C\Proj^2$ are supreme, but there are also manifolds admitting Einstein metrics but no supreme Einstein metric, like $S^2\times S^2$ \cite{LeB99}.

The Yamabe functional is linked to the stability of Einstein metrics through a theorem of Böhm--Wang--Ziller \cite[Thm.~C]{BWZ04}.
\begin{theorem}
\label{thm:localY}
An Einstein metric $g$ is $\EH$-stable if and only if its conformal class $[g]$ is a local maximum of $\Y$.
\end{theorem}
In particular, supreme implies $\EH$-stable, and an Einstein metric which is $\EH$-unstable cannot be supreme. In light of Theorem~\ref{thm:localY}, even if one cannot always hope to find a supreme Einstein metric, stability seems desirable in the search for a ``best'' Einstein metric.

\subsection{Scalar curvature rigidity}

The definition of the Einstein--Hilbert action does not generalize well to non-compact manifolds, as the integral needs not converge. One may still define \emph{linear} stability by restricting to compactly supported variations, but a nonlinear version is not straightforward.

In \cite{DK24}, Dahl--Kröncke study the related concept of \emph{scalar curvature rigidity}. We keep the exposition short here, as scalar curvature rigidity of compact and noncompact manifolds and its relation to dynamical stability under the Ricci flow is the subject of Kröncke's survey \cite{Kr25}.
\begin{definition}[Scalar curvature rigidity]
Let $M$ be a (not necessarily compact) manifold. A Riemannian metric $g\in\Met$ is called \emph{scalar curvature rigid} if there exists a neighborhood $\mathscr{U}\subset\Met$ of $g$ such that there is no metric $\tilde g\in\mathscr{U}$ satisfying
\[(\tilde g-g)|_{M\setminus K}=0,\qquad\Vol(K,\tilde g)=\Vol(K,g)\]
for some compact subset $K\subseteq M$, and such that
\[\scal_{\tilde g}\geq\scal_g\text{ on $M$},\qquad\scal_{\tilde g}(p)>\scal_g(p)\quad\text{for some $p\in M$}.\]
\end{definition}
As it turns out, scalar curvature rigidity is a good replacement for $\EH$-stability in the noncompact case \cite{DK24}, as these notions coincide on compact manifolds.
\begin{theorem}
Let $(M,g)$ be a compact Einstein manifold. Then $g$ is $\EH$-stable if and only if it is scalar curvature rigid.
\end{theorem}

\subsection{Dynamical stability}
\label{sec:dynamic}

Any Einstein metric is in particular a \emph{soliton}, i.e.~a self-similar solution to the \emph{Ricci flow}
\[\frac{\partial}{\partial t}g_t=-2\Ric_{g_t},\]
an important device in geometric analysis. On a compact manifold, one may define the \emph{volume-normalized Ricci flow}
\begin{equation}
\frac{\partial}{\partial t}g_t=-2\Ric_{g_t}+\frac{2\EH(g_t)}{n\Vol(g_t)}g_t,\label{ricciflow}
\end{equation}
for which Einstein metrics are in fact fixed points. One may want to analyze the dynamical behaviour of this flow near the fixed points.

\begin{definition}[Dynamical stability]
An Einstein metric $g$ is called \emph{dynamically stable} if the volume-normalized Ricci flow \eqref{ricciflow} starting sufficiently near $g$ stays near $g$ modulo diffeomorphisms. It is called \emph{dynamically unstable} if there exists an ancient solution $(g_t)_{t\in\left]-\infty,T\right[}$ of \eqref{ricciflow}
such that $g_t$ converges to $g$ modulo diffeomorphisms as $t\to-\infty$.
\end{definition}

The dynamical stability of Einstein metrics was investigated in \cite{Kr20}. 
At least in the
compact case, any Einstein metric is either dynamically stable or unstable. The relation between dynamical stability and scalar curvature rigidity, in particular on noncompact manifolds, is the subject of the survey \cite{Kr25}. On compact manifolds, the situation is as follows.

\begin{theorem}
Let $g$ be an Einstein metric of nonpositive scalar curvature. Then $g$ is dynamically stable if and only if it is $\EH$-stable.
\end{theorem}

The proof of this uses suitable entropy functionals on $\Met_1$ which are nondecreasing along the Ricci flow (the \emph{expander entropy} \cite{FIN05} in the negative case, and Perelman's \emph{$\lambda$-functional} \cite{Per02} in the Ricci-flat case). For any sign of the scalar curvature, dynamical stability of $g$ is equivalent to $g$ being a local maximizer of the appropriate entropy functional, which in turn is closely related to the Einstein--Hilbert and Yamabe functional. For Einstein metrics of positive scalar curvature, an analogous functional is the \emph{shrinker entropy} $\nu$ introduced by Perelman \cite{Per02}, but the criterion to be a local maximizer of $\nu: \Met\to\R$ is more complicated. One may study a linearized version of stability by looking at its second variation \cite{CHI04}.

\begin{theorem}
\label{thm:nu2nd}
Let $g$ be an Einstein metric with Einstein constant $E>0$, not isometric to the standard sphere. The decomposition \eqref{EHdecomp} is still orthogonal for $\nu_g''$. Moreover:
\begin{enumerate}[\upshape(i)]
\item For unit volume constant scalar curvature metrics $\tilde g\in\Csc$, we have
\[\nu(\tilde g)=\frac{n}{2}\log\EH(\tilde g)+\frac{n}{2}(1-\log(2\pi n)).\]
\item For any $f\in C^\infty_\perp(M)$, the second variation $\nu_g''(fg,fg)$ has the same sign as $-\Ltwoinprod{\Delta f-2Ef}{f}$, where $\Delta$ is the Laplace--Beltrami operator.
\end{enumerate}
\end{theorem}

\begin{corollary}
Let $g$ be an Einstein metric with Einstein constant $E>0$, and denote with $\lambda_1(\Delta)$ the first positive eigenvalue of the Laplace--Beltrami operator on $C^\infty(M)$.
\begin{enumerate}[\upshape(i)]
\item If $g$ is $\EH$-stable and $\lambda_1(\Delta)>2E$, then $g$ is dynamically stable.
\item If $g$ is $\EH$-unstable or $\lambda_1(\Delta)<2E$, then $g$ is dynamically unstable.
\end{enumerate}
\end{corollary}

In the presence of eigenfunctions to the eigenvalue $2E$, which together with $\varepsilon(g)$ constitute the \emph{infinitesimal solitonic deformations} of $g$, the situation is not clear. Kröncke \cite{Kr20} found a criterion for instability:

\begin{theorem}
\label{thm:cubicobst}
Let $g$ be an Einstein metric with Einstein constant $E>0$. If there exists $f\in C^\infty(M)$ such that $\Delta f=2Ef$ and $\int_Mf^3|\vol_g|\neq0$, then $g$ is dynamically unstable.
\end{theorem}
We will return to the higher order behavior of $\EH$ around a critical point in Section~\ref{sec:moduli}.

\subsection{The Lichnerowicz Laplacian}

The \emph{linear} stability of an Einstein metric $g$ turns out to be governed by the spectrum of a very natural Laplace-type operator, which we shall now introduce.

For now, let $(M,g)$ be any Riemannian manifold. Let $R\in\Gamma(\Lambda^2T^\ast M\otimes\so(TM))$ be the Riemannian curvature tensor of $g$, in our convention given by
\[R_{X,Y}=[\nabla_X,\nabla_Y]-\nabla_{[X,Y]}.\]
Due to pair-symmetry, this may also be seen as a section of the bundle $\Sym^2\so(TM)$. For any Euclidean vector space $T$, let $q$ be the quantization map
\[q: \Sym^{\leq\bullet}\so(T)\to\Univ^{\leq\bullet}\so(T),\qquad X^{\odot k}\mapsto X^k\]
from the symmetric algebra to the universal enveloping algebra over $\so(T)$. 
This is the isomorphism of filtered vector spaces appearing in the Poincar\'e--Birkhoff--Witt theorem.
It gives rise to a bundle map when we replace $T$ with $TM$. Let $VM$ be any (tensor or spinor) bundle associated to the orthonormal frame bundle (or a spin structure); in particular, it comes with a metric connection, also denoted by $\nabla$, and an action $\rho$ of $\so(TM)$. This naturally extends to an action of $\Univ\so(TM)$, which we also denote by $\rho$. For the purpose of this exposition, let us call such a bundle a \emph{geometric vector bundle} (see \cite{SW19} for a more general definition which takes holonomy reductions into account).

\begin{definition}[Standard curvature endomorphism]
For a geometric vector bundle $VM$, we call the bundle map
\[\K(R,VM)=2\rho(q(R))\in\Gamma(\End VM)\]
the \emph{standard} or \emph{Weitzenböck curvature endomorphism}.
\end{definition}

This is always a symmetric endomorphism with respect to the bundle metric on $VM$ fixed by the $\so(TM)$-action. The factor $2$ here is chosen so that $\K(R,TM)=\Ric_g$. If the bundle is clear from context, we may also just write $\K(R)$. If $(e_i)$ is a local orthonormal basis,
\[\K(R)=\sum_{i<j}\rho(e_i\wedge e_j)\rho(R_{e_i,e_j}).\]
To avoid confusion, we note that the operator $\K(R)$ appears under several notations in the literature, such as $q(R)$ or $\Ric$.

\begin{definition}[Lichnerowicz Laplacian]
On sections of a geometric vector bundle $VM$, the \emph{Lichnerowicz Laplacian} $\LL$ is defined as
\[\LL=\nabla^\ast\nabla+\K(R).\]
\end{definition}
This particular operator shows up in a number of geometric applications; most importantly, it coincides with the Hodge--de Rham Laplacian $\Delta$ on differential forms. It is an elliptic differential operator, and hence has discrete spectrum with finite multiplicities, accumulating only at $+\infty$. Moreover it enjoys many favorable properties: it commutes with parallel bundle maps, hence preserves parallel subbundles. If $g$ is Einstein, $\LL$ also commutes with $\delta: \Sy^2(M)\to\Omega^1(M)$ and its adjoint. In particular, it preserves the space $\TT(M)$ of tt-tensors.

\begin{theorem}
\label{thm:secondvartt}
The second variation of the Einstein--Hilbert functional at an Einstein metric $g$ with Einstein constant $E$ is given by
\[\EH_g''(h,h)=-\frac12\Ltwoinprod{\LL h-2Eh}{h}\]
for all $h\in\TT(M)$.
\end{theorem}

In light of this theorem, linear (semi-)stability of an Einstein metric $g$ amounts to the estimate $\LL>2E$ ($\LL\geq2E$) on tt-tensors. The infinitesimal Einstein deformations of $g$ are exactly the tt-eigentensors to the eigenvalue $2E$,
\[\varepsilon(g)=\ker(\LL-2E)|_{\TT(M)}.\]
Even though the Einstein--Hilbert action is not necessarily well-defined in the case where $M$ is noncompact, one may still define linear (in-)stability using the Lichnerowicz Laplacian, e.g.~$(M,g)$ is called strictly linearly stable if
\[\Ltwoinprod{\LL h}{h}>2E\Ltwoinprod{h}{h}\]
for all compactly supported $h\in\TT(M)$, and similarly for semistable/unstable.

The restriction of $\LL-2E$ to $\TT(M)$ is also sometimes called the \emph{linearized Einstein operator}. It is often expressed as
\[\LL-2E=\nabla^\ast\nabla-2\mathring{R},\]
where $\mathring{R}: \Sym^2T^\ast M\to\Sym^2T^\ast M$ is the \emph{curvature operator of the second kind},
\[(\mathring{R}h)(X,Y)=\sum_ih(R(e_i,X)Y,e_i).\]

In particular, the standard curvature endomorphism on symmetric $2$-tensors for an Einstein metric $g$ can be written as $\K(R) = 2E - 2\mathring{R}$.

The Lichnerowicz Laplacian shows up in two beautiful Weitzenböck formulas: on $\Sy^k(M)$, we have
\begin{equation}
\delta\delta^\ast-\delta^\ast\delta=\LL-2\K(R),\label{WBF1}
\end{equation}
and if $g$ is Einstein,
\begin{equation}
(d^\nabla)^\ast d^\nabla+d^\nabla(d^\nabla)^\ast=\LL-\frac12\K(R)
\label{WBF2}
\end{equation}
on $\Sy^2(M)\subset\Omega^1(M,T^\ast M)$, where $d^\nabla: \Omega^k(M,T^\ast M)\to\Omega^{k+1}(M,T^\ast M)$ is the twisted exterior derivative.

On $\TT(M)$, the first Weitzenböck formula \eqref{WBF1} implies a lower bound for the Lichnerowicz Laplacian in terms of $\K(R)$ \cite[Prop.~6.2]{HMS15}, namely
\begin{equation}\label{estimate}
\LL \ge 2 \K(R)
\end{equation}
with equality exactly on the space of trace-free \emph{Killing tensors} (that is, tensors in the kernel of $\delta^\ast$).
Note that trace-free Killing tensors are automatically divergence-free \cite[Cor.~3.10]{HMS15}.

Calculating or estimating the spectrum of $\LL$ is in general difficult. Restricting to tt-tensors means searching for eigensections of  
the bundle of trace-free symmetric
$2$-tensors 
$VM=\Sym^2_0T^\ast M$ which are at the same time divergence-free. The latter condition is mostly taken care of by the following relation \cite{Sch22a}, which is in parts already implicit in \cite{Koiso82}. 
Denote $\Sy^2_0(M)=\Gamma(\Sym^2_0T^\ast M)$, and 
let $\theta: \Omega^1(M)\to\Sy^2_0(M)$ be the \emph{conformal Killing operator}, defined by
\[\theta\alpha=\delta^\ast\alpha+\frac{2}{n}(d^\ast\alpha)g.\]
Its kernel consists of the one-forms dual to conformal Killing vector fields on $(M,g)$.

\begin{theorem}
If $g$ is an Einstein metric, the sequence
\[0\longrightarrow\ker\theta\stackrel{\subset}{\longrightarrow}\Omega^1(M)\stackrel{\theta}{\longrightarrow}\Sy^2_0(M)\stackrel{P}{\longrightarrow}\TT(M)\longrightarrow 0,\]
where $P$ denotes the $L^2$-orthogonal projection, is exact and every arrow commutes with the Lichnerowicz Laplacian.
\end{theorem}
We note two additional facts: on an Einstein manifold $(M,g)$, Killing one-forms are always eigentensors of the Laplacian to the eigenvalue $2E$, and if $(M,g)$ is in addition not isometric to the standard sphere, then all conformal Killing vector fields are Killing. As a consequence, we can state the following.
\begin{corollary}
\label{thm:ttLL}
Let $(M,g)$ be Einstein with Einstein constant $E$, and not isometric to the standard sphere. Then there are exact sequences
\begin{enumerate}[\upshape(i)]
    \item for the space of IED:
    \[0\longrightarrow\ker\delta^\ast\longrightarrow\ker(\Delta-2E)|_{\Omega^1(M)}\longrightarrow\ker(\LL-2E)|_{\Sy^2_0(M)}\longrightarrow\varepsilon(g)\longrightarrow 0,\]
    \item for any $\lambda\neq2E$:
    \[0\longrightarrow\ker(\Delta-\lambda)|_{\Omega^1(M)}\longrightarrow\ker(\LL-\lambda)|_{\Sy^2_0(M)}\longrightarrow\ker(\LL-\lambda)|_{\TT(M)}\longrightarrow 0,\]
\end{enumerate}
and all the arrows commute with isometries.
\end{corollary}

The value $2E$ that comes up in the discussion of linear stability and infinitesimal deformability is also called the \emph{critical eigenvalue}. Recall that $2E$ is  the lower bound of the Laplace operator on co-closed $1$-forms, for a compact Einstein manifold with $E>0$. This follows from
estimate \eqref{estimate} on $1$-forms. In the case of compact Kähler--Einstein manifolds with positive Ricci curvature it is known that $2E$ is the lower bound for the first eigenvalue of the Laplace--Beltrami operator. Equality is attained if the manifold admits Killing vector fields.



There is another type of stability which occurs in some physical applications, such as Anti-de Sitter product spacetimes and generalized (Schwarzschild--Tangherlini) black holes \cite{GH02,GHP03}.

\begin{definition}[Physical stability]
\label{phys-stab}
An Einstein metric $g$ with Einstein constant $E$ is called \emph{physically stable} if $\LL\geq\frac{9-n}{4}E$ on tt-tensors.
\end{definition}

This is clearly much weaker than linear stability. Later we will see that this bound, which comes from the \emph{Breitenlohner--Freedman bound} in relativistic quantum field theory, exactly appears on manifolds with Killing spinors (Theorem~\ref{thm:killingphys}).

The Lichnerowicz Laplacian also features in the study of \emph{Ricci iterations}. A Ricci iteration is a sequence of metrics $(g_i)_{i=1}^\infty$ satisfying
\[\Ric_{g_{i+1}}=g_i.\]
This constitutes a discrete analogue of the Ricci flow. As Buttsworth--Hallgren \cite[Thm.~4.1]{RicIt} have shown, a sufficient condition for the existence and dynamical stability of a Ricci iteration starting at a constant multiple of $g$ is that $\ker\LL|_{\Sy^2(M)}=\R g$ and that $\LL|_{\ker\delta}$ has no other eigenvalues in the interval $[-2E,2E]$.

\subsection{Curvature conditions for stability}

We conclude this section by listing some sufficient conditions for stability in terms of the curvature. The first criterion goes back to Koiso \cite{Koiso78} and is a consequence of the two Weitzenböck formulas \eqref{WBF1} and \eqref{WBF2}.

\begin{theorem}
Let $g$ be an Einstein metric with Einstein constant $E$. If pointwise \[\K(R,\Sym^2_0T^\ast M)>\min\{E,4E\},\]
or equivalently $\mathring{R}<\max\{-E,E/2\}$ on $\Sym^2_0T^\ast M$, then $g$ is strictly stable.
\end{theorem}

Fujitani \cite{Fuj79} worked out how the eigenvalues of $\mathring{R}$ are controlled by the sectional curvature. This yields the following corollary, originally due to Bourguignon.

\begin{corollary}
If the sectional curvature of an Einstein metric $g$ is negative or $\kappa$-pinched with $\kappa>\frac{n-2}{3n}$, then $g$ is strictly stable.
\end{corollary}

We emphasize that the condition $n\geq3$ is crucial for the negative sectional curvature case, as hyperbolic surfaces admit IED.

In any of the statements above, replacing the strict inequalities with $\geq$ still yields semistability. In \cite{Kr15b}, Kröncke further characterizes these cases under the assumption $\varepsilon(g)\neq0$. Another result by Koiso \cite{Koiso78}, which uses the estimate $\mathring{R}<-E$ directly, is as follows.

\begin{corollary}
Let $(M,g)$ be an (Einstein) locally symmetric space of noncompact type. If $M$ has no local two-dimensional factor, then $g$ is strictly stable.
\end{corollary}

For symmetric spaces of compact type the situation is more delicate, but ultimately understood through the help of harmonic analysis; we come to this in Section~\ref{sec:homogeneous}.

On oriented four-manifolds, Fine--Krasnov--Singer \cite{FKS21} give a criterion for the infinitesimal non-deformability of an Einstein metric in terms of the curvature operator $\widehat R_+=\widehat R|_{\Lambda^2_+}$ on self-dual two-forms, using the strict stability for a different action functional whose critical points are Einstein metrics. Here we use the sign convention where positivity of the curvature operator $\widehat R$ implies positive sectional curvature, and we note that for Einstein metrics, $\widehat R$ preserves the splitting $\Lambda^2T^\ast M=\Lambda^2_+\oplus\Lambda^2_-$ into self-dual and anti-self-dual two-forms.

\begin{theorem}
Let $(M^4,g)$ be an oriented Einstein manifold. If $\widehat R_+<0$, then $\varepsilon(g)=0$.
\end{theorem}

Further stability conditions in terms of the Weyl curvature or, in the Kähler--Einstein case, the Bochner tensor, are contained in \cite{IN04,Kr15a,BK25}.

The above results motivate the guess that in the negative curvature regime, Einstein metrics tend to be stable.
Moreover, all known compact, Ricci-flat Einstein manifolds are $\EH$-stable (see Section~\ref{sec:parallelspinor}).

\begin{questype}{Conjecture}
\label{qu:positivemass}
Compact Einstein manifolds of nonpositive scalar curvature are $\EH$-stable.
\end{questype}

For Ricci-flat metrics, the linearized version is called the \emph{positive mass problem} \cite{Dai07}.

\section{The moduli space of Einstein metrics}
\label{sec:moduli}

How many Einstein metrics are there on a given compact manifold $M$? This is a difficult problem in general. Of course, it is sensible to count metrics $g_1,g_2$ which are \emph{isometric} (that is, $g_2=\varphi^\ast g_1$ for some diffeomorphism $\varphi\in\Diff$) or \emph{homothetic} (that is, $g_2=c^2g_1$ for some $c\in\R^\times$) as the same. Without restriction, we can homothetically scale any metric on $M$ so that it has total volume one. Quotienting the set of Einstein metrics on $\Met_1$ by the group of diffeomorphisms then leads to the \emph{moduli space} of Einstein metrics.

\begin{definition}[Einstein moduli space]
The quotient
\[\EMod=\{g\in\Met_1\,|\,g\text{ is Einstein}\}/\Diff\]
is called the \emph{Einstein moduli space} of $M$.
\end{definition}

Some authors instead take the moduli space to be the quotient by the group $\Diff^0$ of diffeomorphisms isotopic to the identity, that is, they distinguish between metrics which are not related by an element of $\Diff^0$. This approach has advantages in some situations, but for our treatment it does not matter.

We have a good understanding of $\EMod$ on Riemann surfaces (where it coincides with the moduli space of complex structures), on flat manifolds of dimension $3$ or $4$ (where it coincides with the moduli space of flat metrics \cite[\S12.B]{besse}), and on the K3 surface (where it coincides with the moduli space of Calabi--Yau metrics, see Section~\ref{sec:parallelspinor}). For a more detailed picture in dimension four, see the survey \cite{anderson} by Anderson.

In most cases, understanding the entire moduli space at once is out of reach. One may still ask whether, given some Einstein metric on $M$, there are other Einstein metrics nearby or not.

\begin{definition}[Rigidity]
An Einstein metric $g\in\Met_1$ is called \emph{rigid} if its isometry class $[g]$ is isolated in $\EMod$.
\end{definition}

Since equivalence classes of metrics are unwieldy objects, for a local study of the moduli space it is preferable to pass to a slice to the action of $\Diff$, that is, a suitable submanifold of $\Met_1$ tangent to the second summand in decomposition~\eqref{diffdecomp}. This is achieved by Ebin's Slice Theorem.

\begin{theorem}[Slice Theorem]
Fix $g\in\Met_1$. There exists a (real analytic) ILH-submanifold $\Slice_g\subset\Met_1$, called the \emph{Ebin slice} to the action of $\Diff$, which satisfies the following properties.
\begin{enumerate}[\upshape(i)]
    \item $\Slice_g$ is invariant under the isometry group $\Isom(g)$.
    \item If $\varphi\in\Diff$ satisfies $\varphi^\ast\Slice_g\cap\Slice_g\neq\emptyset$, then $\varphi\in\Isom(g)$.
    \item There is a local section $\chi: \Diff/\Isom(g)\dashrightarrow\Diff$, defined in a neighborhood of the identity coset, such that the local map
    \[\Diff/\Isom(g)\times\Slice_g\dashrightarrow\Met_1,\qquad (\bar\varphi,\tilde g)\mapsto\chi(\bar\varphi)^\ast\tilde g\]
    is a diffeomorphism onto a neighborhood of $g\in\Met_1$.
    \item $T_g\Slice_g=\ker\delta$.
\end{enumerate}
In particular, $\Slice_g/\Isom(g)$ is homeomorphic to a neighborhood of $[g]\in\Met_1/\Diff$.
\end{theorem}

For the following arguments, we fix some choice of slice $\Slice_g$.

\begin{definition}[Premoduli space]
Let $g$ be an Einstein metric. The subset $\widetilde\EMod_g$ of Einstein metrics in $\Slice_g$ is called the \emph{premoduli space} of Einstein metrics around $g$.
\end{definition}

The quotient $\widetilde\EMod_g/\Isom(g)$ is then a neighborhood of $[g]\in\EMod$. One may hope that the premoduli space around an Einstein metric $g$ is a manifold with tangent space $\varepsilon(g)$, but we will see that this is not the case in general.
\begin{theorem}
\label{thm:modulinonpos}
If $g$ is an Einstein metric with nonpositive scalar curvature, then the identity component $\Isom^0(g)$ acts trivially on $\widetilde\EMod_g$.
\end{theorem}
Thus, in the nonpositive scalar curvature case, if the premoduli space turns out to be a manifold, the moduli space will be an \emph{orbifold} with local group $\Isom(g)/\Isom^0(g)$.

\subsection{Properties of the moduli space}

The global and even local structure of the Einstein moduli space is still poorly understood in many cases, although there has been some progress in the last decades. Since Einstein metrics are critical points of the Einstein--Hilbert functional, the connected components of $\EMod$ are separated by the Einstein constants.

\begin{theorem}
\label{thm:EMod}
The Einstein moduli space has the following properties.
\begin{enumerate}
    \item $\EMod$, like $\Met_1/\Diff$, is Hausdorff.
    \item $\EH$ is locally constant on $\EMod$.
    \item $\EMod$, like $\Met_1$, has a countable basis of topology, and therefore at most countably many connected components.
\end{enumerate}
\end{theorem}

There exist examples where the moduli space is connected (see Section~\ref{sec:parallelspinor}), but also where it has infinitely many connected components \cite{WZ90,Coe12}. For more information on the set of Einstein constants, or the (local) dimension of the moduli space, see \cite[\S12 and Add.~A]{besse}. In this survey, we primarily focus on the local behavior of $\EMod$ around an Einstein metric.

In some cases, the moduli space is (locally) a manifold or orbifold, but this needs not generally be so. However, currently no example is known where the moduli space has a worse than orbifold singularity.

\begin{questype}{Open question}
Is the Einstein moduli space always (locally) an orbifold?
\end{questype}

Let us introduce a (nonlinear) operator on $\Met_1$ whose zeros are precisely Einstein metrics, and which is agnostic about the Einstein constant.

\begin{definition}[Einstein operator]
The operator $\Eop: \Met_1\to\Sy^2(M)$ given by
\[\Eop(g)=\Ric_g-\frac{\EH(g)}{n}g\]
is called the \emph{Einstein operator}.
\end{definition}

Restricting to the Ebin slice $\Slice_g$ around an Einstein metric $g$, the premoduli space around $g$ is then the zero set of the real analytic map $\Eop$. In fact, the IED of $g$ constitute the kernel of the linearization of $\Eop$.

\begin{theorem}
\label{thm:einsteinop}
Let $g$ be an Einstein metric with Einstein constant $E$.
\begin{enumerate}[\upshape(i)]
    \item The first variation $\Eop'_g: \Sy^2(M)\to\Sy^2(M)$ of $\Eop$ at $g$ is given by
    \[\Eop'_g=\frac12\left(\LL-2E\right)-\delta^\ast\beta,\]
    where $\beta=\delta+\frac12d\tr: \Sy^2(M)\to\Omega^1(M)$ is the \emph{Bianchi operator}.
    \item $T_g\Slice_g\cap\ker\Eop'_g=\varepsilon(g)$.
\end{enumerate}
\end{theorem}

Even though there is an implicit function theorem for ILH-manifolds, it cannot be applied to $\Eop$ in order to show that $\widetilde\EMod_g$ is a manifold, because $\Eop'_g$ fails to be surjective. 
Indeed, the second Bianchi identity implies that $\beta(\Eop(g))=0$ for all $g\in\Met_1$; thus $\im\Eop'_g\subset\ker\beta$ whenever $g$ is Einstein. 
One may, however, apply the implicit function theorem to $\Eop$ composed with the orthogonal projection to $\im\Eop'_g$, to obtain a real analytic submanifold $Z\subset\Slice_g$ with $T_gZ=\varepsilon(g)$ which contains $\widetilde\EMod_g$ as a real analytic subset \cite{Koiso80}.

As a consequence of the above discussion and the properties of real analytic sets, we have gained some insight into the structure of the moduli space.

\begin{theorem}
\label{thm:arcwise}
\begin{enumerate}[\upshape(i)]
    \item $\EMod$ is locally arcwise connected (where \emph{arc} means a real analytic curve).
    \item If $\varepsilon(g)=0$ for an Einstein metric $g$, then $g$ is rigid.
\end{enumerate}
\end{theorem}

\subsection{Deformation theory}
\label{sec:deformations}

By the arcwise connectedness of $\EMod$ (Theorem~\ref{thm:arcwise}), if an Einstein metric is not rigid, there exists a real analytic curve of Einstein metrics through $g$ which do not all represent the same class in $\EMod\subset\Met_1/\Diff$. Thus it remains to study the existence of these curves.

\begin{definition}[Einstein deformation]
Let $g\in\Met_1$ be an Einstein metric. An \emph{Einstein deformation} of $g$ is a smooth curve of Einstein metrics $(g_t)$ in $\Met_1$ with $g_0=g$. It is called \emph{nontrivial} if the metrics $g_t$ do not all represent the same class in $\EMod\subset\Met_1/\Diff$.
\end{definition}

Results in complex geometry guarantee the existence of many examples of nontrivial curves of Kähler--Einstein metrics from deformations of the complex structure. Near a negative or Ricci-flat Kähler--Einstein metric, the Aubin--Calabi--Yau theorems guarantee the existence of a Kähler--Einstein metric compatible with a given complex structure \cite[\S11]{besse}; in the positive case, this works if the underlying complex manifold is $K$-polystable, due to the resolution of the Donaldson--Tian--Yau conjecture \cite{CDS,Tian15}. In dimension four, this condition is equivalent to the vanishing of the Futaki invariant \cite{Tian90}.

Leveraging the connection between Sasaki--Einstein and Kähler--Einstein manifolds, several authors \cite{BGK05,CS19,LST24} have recently shown the existence of high-dimensional families of Sasaki--Einstein structures on various homotopy spheres, as well as Sasaki--Einstein deformations of toric 3-Sasaki $7$-manifolds \cite{Coe17}. For (noncompact) solvmanifolds, high-dimensional families of Einstein metrics with negative sectional curvature have been constructed by Heber \cite{Heber}.

Given an \emph{Einstein deformation} of $g$, that is, a smooth curve of Einstein metrics $(g_t)$ in $\Met_1$ with $g_0=g$, its first order jet $h=\dot g_0$ will satisfy the linearized Einstein equation
\[\Eop'_g(h)=\frac{\D}{\D t}\Big|_{t=0}\Eop(g_t)=0.\]
By Theorem~\ref{thm:einsteinop}, the first order jets to Einstein deformations which are orthogonal to the orbit of $\Diff$ are infinitesimal Einstein deformations. However, not every IED necessarily corresponds to an actual Einstein deformation.

\begin{definition}[Integrability]
Let $g$ be Einstein, and $h\in\varepsilon(g)$.
\begin{enumerate}[\upshape(i)]
    \item The IED $h$ is called \emph{integrable} if there exists an Einstein deformation $(g_t)$ such that $g_0=g$ and $\dot g_0=h$.
    \item For any $k\geq2$, the IED $h$ is called \emph{formally integrable to order $k$} if there exist $h_2,\ldots,h_k\in\Sy^2(M)$ such that
    \begin{equation}
    \frac{\D^j}{\D t^j}\Big|_{t=0}\Eop\left(g+th+\sum_{i=2}^j\frac{t^i}{i!}h_i\right)=0
    \label{integrable}
    \end{equation}
    for $2\leq j\leq k$.
\end{enumerate}
\end{definition}

By an approximation argument of M.~Artin, an IED is integrable if and only if it is formally integrable to every order $k\geq2$.

Integrability of IED also has implications for stability. The following theorem is implicitly stated in \cite{DK24}; for the proof, we adapt an argument of \cite{DWW05}.
\begin{theorem}
\label{thm:stablefromintegrable}
Let $g$ be a semistable Einstein metric such that all of its IED are integrable. Then $g$ is $\EH$-stable.
\end{theorem}
\begin{proof}
By the assumption that all IED are integrable, $\widetilde\EMod_g$ is a submanifold of $\Slice_g$ with $T_g\widetilde\EMod_g=\varepsilon(g)$. Consider the Einstein--Hilbert functional restricted to $\Csc\cap\Slice_g$. First, it is constant on the submanifold $\widetilde\EMod_g\subset\Csc\cap\Slice_g$, say $\EH|_{\widetilde\EMod_g}\equiv c$. Second, we have $T_g(\Csc\cap\Slice_g)=\TT(M)$, so by the assumption of semistability and the fact that $\varepsilon(g)$ is precisely the null space of $\EH''_g|_{\TT(M)}$, the second variation $\EH''_g$ is negative-definite on the normal bundle of $\widetilde\EMod_g\subset\Csc\cap\Slice_g$. Thus there is a neighborhood $\mathscr{U}\supset\widetilde\EMod_g$ in $\Csc\cap\Slice_g$ such that $\EH|_{\mathscr{U}\backslash\widetilde\EMod_g}<c$. So $g$ is a local maximum of $\EH$ on $\Csc\cap\Slice_g$, and by diffeomorphism invariance also on $\Csc$.
\end{proof}

How to determine whether an IED is integrable or not? The integrability conditions \eqref{integrable} can successively be stated as
\[\Eop'_g(h_k)+P^k(h,h_2,\ldots,h_{k-1})=0,\qquad k\geq2,\]
for some polynomials $P^k$ of degree $k$, which depend on up to the $k$-th derivative of $\Eop$ at $g$, and which satisfy $\beta\circ P^k=0$. The composition of $P^k$ with the quotient map $\ker\beta\to\ker\beta/\im\Eop'_g$ is called the \emph{$k$-th order obstruction} to integrability, and the quotient $\ker\beta/\im\Eop'_g$ is called the \emph{obstruction space}. What makes the deformation theory of Einstein metrics so difficult is that the obstruction space coincides with the deformation space itself:

\begin{theorem}
\label{thm:obstspace}
If $g$ is an Einstein metric, the subspace $\ker\beta\subset T_g\Met_1$ splits orthogonally as
\[\ker\beta=\im\Eop'_g\oplus\varepsilon(g).\]
\end{theorem}

Thus we expect that IED are often obstructed. Indeed, there are examples of Einstein manifolds with a nontrivial space of IED, none of which are integrable. The first example of this was the symmetric Einstein metric on $\C\Proj^{2k}\times\C\Proj^1$ \cite{Koiso82}; we will exhibit more such examples (all homogeneous of positive scalar curvature) in Section~\ref{sec:symmetric}.


In light of the examples where we know about the existence or nonexistence of nontrivial Einstein deformations, one might speculate that the integrability of IED is related to the existence of Killing fields (cf.~\cite{anderson}). 
In the Lorentzian setting \cite{FM73,Mon75,Mon76}, it is known that the absence of Killing fields implies integrability of IED. 
Inspired by this approach, Ozuch \cite{Ozu21} has studied integrability of IED in dimension four.

\begin{questype}{Open question}
If an Einstein metric has no Killing vector fields, are all of its IED integrable?
\end{questype}

Implicit function arguments applied to the linearization $\Eop_g'$ of the Einstein operator are also at the heart of the study of desingularization of Einstein orbifolds \cite{BiqI,BiqII}. The completion of the Einstein moduli space $\EMod$ in the Gromov--Hausdorff distance adds a ``boundary'' which consists, at least in dimension four, of Einstein orbifolds. These particular Einstein orbifolds are \emph{desingularizable}, that is, they can be deformed into non-singular Einstein manifolds. A recent achievement is the result that not every Einstein 4-orbifold is desingularizable \cite{Ozu21}.

Lastly, we note that there is a similar deformation theory for Ricci solitons, which is due to Podestà--Spiro and Kröncke \cite{PS15,Kr16}. The analogous infinitesimal deformations are the \emph{infinitesimal solitonic deformations} mentioned in Section~\ref{sec:dynamic}. However, we will not say more about this, as it is beyond the scope of this survey.

\subsection{The second order obstruction}

To conclude this section with something more concrete, let us take a closer look at the obstruction to second order. Equation~\eqref{integrable} gives the condition
\[0=\frac{\D^2}{\D t^2}\Big|_{t=0}\Eop\left(g+th+\frac{t^2}{2}h_2\right)=\Eop'_g(h_2)+\Eop''_g(h,h).\]
Due to Theorem~\ref{thm:obstspace}, we can encode the second order obstruction as
\[\Psi: \varepsilon(g)\times\varepsilon(g)\times\varepsilon(g)\to\R,\qquad \Psi(\alpha, \beta, \gamma)=2\Ltwoinprod{\Eop_g''(\alpha,\beta)}{\gamma}.\]
It turns out that the trilinear form $\Psi$ is symmetric \cite{KS24}, thus we may also write it as a cubic form $\Psi(h)=\Psi(h,h,h)$. Koiso \cite{Koiso82} gives an explicit expression for it, and Nagy--Semmelmann \cite{NS25} write it in terms of the Frölicher--Nijenhuis bracket
\[\FN{\,\cdot\,}{\cdot\,}: \Omega^1(M,TM)\times\Omega^1(M,TM)\to\Omega^2(M,TM).\]

\begin{theorem}
\label{thm:secondobst}
Let $g$ be Einstein, with Einstein constant $E$, and let $h, \alpha, \beta, \gamma \in\varepsilon(g)$.
\begin{enumerate}[\upshape(i)]
    \item $\Psi$ is related to the third variation of the Einstein--Hilbert functional by
    \[\EH'''_g(\alpha, \beta, \gamma)=-\Ltwoinprod{\Ric_g''(\alpha,\beta)}{\gamma}=-\frac12\Psi(\alpha,\beta , \gamma).\]
    \item The cubic form $\Psi$ may be written as
    \begin{align*}
    \Psi(h)&=\int_M(2Eh\indices{_i^j}h\indices{_j^k}h\indices{_k^i}+(\nabla_i\nabla_jh_{kl})(3h^{ij}h^{kl}-6h^{ik}h^{jl})|\vol_g|\\
    &=\int_M(3\langle\FN{h}{h},d^\nabla h\rangle-E\tr(h^3))|\vol_g|.
    \end{align*}
\end{enumerate}
\end{theorem}

Concretely, an IED $h\in\varepsilon(g)$ is integrable to second order if and only if $\Psi(h,h,\alpha)$ vanishes for all $\alpha \in\varepsilon(g)$, that is, $h$ is a critical point of the cubic form $\Psi$ on $\varepsilon(g)$.

From Theorem~\ref{thm:secondobst} (i) and the Taylor expansion of $\EH$, we can conclude a criterion for instability analogous to Theorem~\ref{thm:cubicobst}, cf.~\cite{Kr15c,HSS25b}.
\begin{corollary}
\label{thm:nonlinearunstable}
Let $g$ be an Einstein metric. If the second order obstruction $\Psi$ does not vanish identically, then $g$ is $\EH$-unstable.
\end{corollary}

\section{Products and fibrations}
\label{sec:prodfib}

\subsection{Stability of products}
The Riemannian product of two Einstein manifolds is again Einstein if both factors have the same Einstein constant $E$. In this case $E$ is also the Einstein constant of the product metric. The stability of products of Einstein manifolds was studied in the case $E<0$ in \cite{AM} and for the general case in \cite{Kr15b}. We summarize the results.

\begin{theorem}
    \begin{enumerate}[\upshape(i)]
        \item The product of two semistable Einstein manifolds with Einstein constant $E\le 0$ is  again semistable.
        \item The product of two Einstein manifolds $(M^m,g_M)$ and
$(N^n, g_N)$ with $E>0$ is linearly unstable. One destabilizing direction is always $h:= ng_M - mg_N$.
    \end{enumerate}
\end{theorem}

It is in fact possible to write the complete spectrum of $\LL|_{\TT(M\times N)}$ for a product metric in terms of the spectrum of the Laplace operators on functions, $1$-forms and symmetric $2$-tensors on the factors \cite{Kr15b}. In particular, the product metric on $M\times N$ has IED if $2E$ is an eigenvalue of the Laplace operator on functions on $M$ or $N$, which is e.g.~the case for the complex projective space $\C\Proj^m$. The non-integrability of these particular IED on $\C\Proj^{2m}\times\C\Proj^1$ was already mentioned in Section \ref{sec:deformations}.

\subsection{Stability of warped products}

Warped product metrics often come up in practice (e.g.~cosmological models) and are relatively computationally accessible. The stability of warped product Einstein metrics 
has been the subject of several articles 
\cite{D,HHS,Kr17}. The following results were proved by Kröncke in \cite{Kr17}.

\begin{theorem}\label{thm:kr1}
Let $(M, g)$ be a Ricci-flat manifold. Then the warped product Einstein manifold $(\tilde M, \tilde g) = (\R \times M, dr^2 + e^{2r} g )$ is semistable if and only if $(M, g)$ is semistable. In this case the Einstein manifold $(\tilde M, \tilde g)$ is strictly stable.
\end{theorem}

\begin{theorem}\label{thm:kr2}
Let $(M^n,g)$ be a compact Einstein manifold with $E=n-1$. Then the Ricci-flat cone $(\R_+ \times M, dr^2 + r^2g)$ is semistable if and only if $(M^n,g)$ is physically stable. The same is true for the hyperbolic cone $(\R_+ \times M, dr^2 + \sinh^2(r) g)$.
\end{theorem}

The \emph{sine-cone} over a compact Einstein manifold $(M^n,g)$ with $E=n-1$ is the Einstein manifold $(\left]0,\pi\right[\times M,dr^2+\sin^2(r)g)$, which has Einstein constant $n$. The stability and deformability of these sine-cones are investigated in \cite{Kr21}. 
Note that the
above cone metrics are allowed to be singular at the cone points. 
A more general study of the stability of warped products was undertaken by Batat--Hall--Murphy \cite{BHM17}.

\subsection{Stability in Riemannian submersions}
\label{sec:subm}

A further generalization is the case of a Riemannian submersion. Assume that $M$ is connected, and let $\pi: (M,g_M)\to (B,g_B)$ be a Riemannian submersion with totally geodesic fibers. Then the fibers are all isometric to some ``model fiber'' $(F,g_F)$. If $g_M$ is Einstein, then the scalar curvatures $\scal_F$ and $\scal_B$ of $g_F$ resp.~$g_B$ are necessarily constant \cite[\S9F]{besse}.
Going forward, we also require that $(M,g_M)$ is not locally isometric to the product $(B,g_B)\times(F,g_F)$.

The stability of positive Einstein metrics on the total space of a Riemannian submersion with totally geodesic fibers was studied by C.~Wang and M.~Wang \cite{WW}.

\begin{theorem}\label{thm:submersion1}
Let $\pi: (M^m, g_M) \rightarrow (B^n, g_B)$ be a Riemannian submersion with totally geodesic fibers which is not locally a product  and let $g_M$ be a positive Einstein metric. If $2 \frac{\scal_F}{m-n} > \frac{\scal_B}{n}$, then $g_M$ is unstable and the tt-tensor $h= n g_M - m \pi^\ast g_B$ is a destabilizing direction.
\end{theorem}

This destabilizing direction is always a Killing tensor and can be explained as coming from the \emph{canonical variation} $(g_t)$, the family of metrics on $M$ given by
\[g_t=\pi^\ast g_B+t(g_M-\pi^\ast g_B),\qquad t>0.\]
In the special case where $g_F$ and $g_B$ are both Einstein with positive Einstein constants $E_F=\frac{\scal_F}{m-n}$ resp.~$E_B=\frac{\scal_B}{n}$, the condition $2 E_F\neq E_B$ implies the existence of a unique second Einstein metric in the canonical variation \cite[Lem.~9.74]{besse}. At least one of the two Einstein metrics will be unstable, as one can see from the graph of the function $t\mapsto\EH(g_t)$ \cite[Prop.~9.72]{besse}.

In the case where the base $(B,g_B)$ is a Riemannian product, Wang--Wang \cite{WW} showed that pullbacks of certain tt-tensors on $B$ provide destabilizing directions on the total space. An interesting example are the Einstein metrics constructed by Wang--Ziller \cite{WZ90} on principal torus bundles over products of Kähler--Einstein manifolds with positive scalar curvature. By results of Böhm \cite{Boe05} and Wang--Wang \cite{WW}, these examples show that the coindex and nullity of $\EH_g''|_{\TT(M)}$ of an irreducible Einstein manifold $(M,g)$ can be arbitrarily high.

\begin{theorem}\label{thm:fibration}
Let $p,r\in\N$. There exists a Riemannian submersion with totally geodesic fibers $(M,g)\to \prod_{i=1}^p(B_i,g_i)$ which is a principal $T^r$-bundle, where $(B_i,g_i)$ are Kähler--Einstein manifolds with positive scalar curvature, such that $g$ is Einstein and
\begin{enumerate}[\upshape(i)]
    \item $\dim\varepsilon(g)+\coind(g)\geq p-1$,
    \item $\coind(g)\geq p-r$.
\end{enumerate}
\end{theorem}

A simple special case studied by C.~Wang \cite{W} are Sasaki--Einstein manifolds which are \emph{regular}, that is, they are principal circle bundles over Kähler--Einstein manifolds of positive scalar curvature. In these cases, it is possible to produce destabilizing directions by pulling back suitable tt-eigentensors from the base, see Section~\ref{sec:SE}. A more general study of the stability of principal circle bundles is carried out in \cite{WW}.

\section{Manifolds with a parallel spinor}
\label{sec:parallelspinor}

As already indicated in the preceding sections, every sign of the Einstein constant has its own peculiarities. Let us first focus on $E=0$, that is, Ricci-flat manifolds.

In this section, let $M$ be a compact, connected manifold of dimension $n\geq3$. It is known as the structure theorem for Ricci-flat manifolds \cite{FW} that if $g$ is a Ricci-flat metric on $M$, then $(M,g)$ is finitely covered by the Riemannian product of a \emph{simply connected}, Ricci-flat manifold and a flat torus.
The possible holonomy groups of an irreducible, simply connected, Ricci-flat manifold $(M,g)$ are listed in Table~\ref{tab:holonomy}. Except for the case of generic holonomy, the holonomy group is simply connected; 
hence it lifts to a subgroup of $\Spin(n)$ and defines a spin structure inducing the metric $g$.
 Conversely, a simply connected spin manifold with a parallel spinor always has special holonomy and is Ricci-flat.

\begin{table}[h]
    \centering
    \renewcommand{\arraystretch}{1.2}
    \begin{tabular}{|c|c|c|}\hline
    $\Hol(M^n,g)$&Description&Dim.~parallel spinors\\\hline\hline
    $\SO(n)$&generic holonomy&--\\\hline
    $\SU(m)$, $n=2m$&Calabi--Yau&$2$\\\hline
    $\Sp(m)$, $n=4m$&hyper-Kähler&$m+1$\\\hline
    $\rmG_2$, $n=7$&torsion-free $\rmG_2$&$1$\\\hline
    $\Spin(7)$, $n=8$&torsion-free $\Spin(7)$&$1$\\\hline
    \end{tabular}
    \caption{Holonomy groups, geometric meaning, and dimension of the space of parallel spinors.}
    \label{tab:holonomy}
\end{table}

In fact, so far every known compact, Ricci-flat manifold is covered by one admitting a parallel spinor (and therefore has special holonomy.)

\begin{questype}{Open question}
Does there exist a compact, Ricci-flat manifold with generic holonomy?
\end{questype}

There also seem to be no general results about the stability and rigidity of Ricci-flat manifolds without assuming the existence of a parallel spinor. Consequentially, in this section we will only treat manifolds admitting a parallel spinor, possibly up to a finite cover.

Stability as well as the moduli space of metrics with parallel spinors have been investigated by M.~Wang \cite{Wang91}, Dai--Wang--Wei \cite{DWW05}, Nordström \cite{Nor13}, and Ammann et al.~\cite{holrig}.

\begin{theorem}
\label{thm:parallelsemistable}
If $(M,g)$ is a (compact or noncompact) spin manifold admitting a parallel spinor, then $g$ is semistable.
\end{theorem}

The proof rests on the fact that every parallel spinor $\psi\in\Gamma(\Spinor M)$ defines a parallel bundle map
\[\Sym^2T^\ast M\to T^\ast M\otimes\Spinor M,\qquad h\mapsto(X\mapsto h(X)\cdot\psi).\]
One may calculate that under this map, $\LL$ restricted to tt-tensors corresponds to the square of the (twisted) Dirac operator, which is positive semidefinite.

Let us turn to Einstein deformations of a metric $g$ admitting parallel spinors (up to a finite cover). In \cite{Wang91,DWW05,holrig} it is shown that the conjugacy class of the holonomy group in $\GL(n,\R)$ is locally constant on the Einstein moduli space $\EMod$ (\emph{holonomy rigidity}). In particular, the connected component of $g$ in the Einstein moduli space $\EMod$ consists of metrics also admitting a parallel spinor (up to a finite cover).

Goto \cite{Goto} studied moduli spaces of torsion-free $G$-structures, where $G$ is one of $\SU(m)$, $\Sp(m)$, $\rmG_2$ or $\Spin(7)$. A key insight is that infinitesimal deformations of such structures are always integrable, and that their moduli spaces (defined as quotients by $\Diff^0$) are manifolds. This generalizes earlier results for $\SU(m)$ (a consequence of the work of Koiso~\cite{Koiso83} and the Bogomolov--Tian--Todorov theorem) and $\rmG_2$ (due to theory by Bryant, Harvey and Joyce). As Nordström showed \cite{Nor13}, the map
\[\{\text{torsion-free $G$-structures}\}/\Diff^0\longrightarrow\{\text{metrics with holonomy $G$}\}/\Diff^0\]
that sends a $G$-structure to the associated metric, is a submersion, and the right hand side is a smooth manifold. It follows that the connected component of $g$ in $\EMod$ is an orbifold, as $\Diff/\Diff^0$ is finite.

By a result of Koiso \cite{Koiso83}, the dimension of the Einstein moduli space around a Calabi--Yau metric $g$ can be written in terms of its Hodge numbers.

\begin{theorem}
\label{thm:calabiyau}
If $(M,g)$ is a Calabi--Yau manifold, then
\[\dim\varepsilon(g)=\dim\widetilde\EMod_g=h^{1,1}-1+2h^{m-1,1}-2h^{2,0}.\]
\end{theorem}
For the relation between deformations of the Einstein and the complex structure, see the discussion in Section~\ref{sec:einsteincomplex}. We note that $h^{2,0}=0$ if $\Hol(M,g)=\SU(m)$, $m\geq3$, and $h^{2,0}=1$ if $\Hol(M,g)=\Sp(m)$ \cite{Joyce}.

The moduli space of torsion-free $\rmG_2$ structures on a given compact $7$-manifold $M$ has dimension $b_3(M)$ \cite{Joyce}, and the tangent space at a given $\rmG_2$-structure (the space of infinitesimal $\rmG_2$-deformations) is its space of harmonic $3$-forms. We note that if $M$ admits a metric of holonomy equal to $\rmG_2$,
with $\rmG_2$-structure $\sigma\in\Omega^3(M)$, then all harmonic $3$-forms on $(M,g,\sigma)$ are of the form $h\diamond\sigma$ for some $h\in\Sy^2(M)$, where $\diamond$ denotes the infinitesimal $\gl(TM)$-action on tensors; the corresponding infinitesimal deformations of the metric $g$ are exactly $h\diamond g=2h$. Imposing the volume constraint $\tr_g h=0$, we obtain:

\begin{theorem}
\label{thm:g2}
If $(M,g)$ is a torsion-free $\rmG_2$-manifold, then
\[\dim\varepsilon(g)=\dim\widetilde\EMod_g=b_3(M)-1.\]
\end{theorem}

Quite similarly, for torsion-free $\Spin(7)$-structures on a compact, oriented $8$-manifold $M$ admitting a metric of holonomy $\Spin(7)$, the moduli space has dimension $b_4^-(M)+1$, where $b_4^-(M)$ is the dimension of the space \emph{anti-self-dual} harmonic $4$-forms for a given $\Spin(7)$-structure $\Omega$. Again, the infinitesimal $\Spin(7)$-deformations are all of the form $h\diamond\Omega$ for some $h\in\Sy^2(M)$, thus:

\begin{theorem}
\label{thm:spin7}
If $(M,g)$ is a torsion-free $\Spin(7)$-manifold,
\[\dim\varepsilon(g)=\dim\widetilde\EMod_g=b_4^-(M).\]
\end{theorem}

We turn to flat manifolds. It is well known that every such manifold is isometric to a quotient $\R^n/G$, where $G$ is a \emph{Bieberbach subgroup} of the group of affinities $\Orth(n)\ltimes\R^n$. Any flat metric is clearly semistable because $\LL=\nabla^\ast\nabla$ is a nonnegative operator. Kröncke \cite{Kr15b} studied the IED of such manifolds.

\begin{theorem}
Let $(M,g)$ be flat, and suppose the holonomy representation decomposes into irreducible representations as
\[\R^n\cong(\rho_1)^{n_1}\oplus\ldots(\rho_k)^{n_k}.\]
The IED of $g$ are precisely the parallel sections of $\Sym^2_0T^\ast M$, and
\[\dim\varepsilon(g)=\sum_i\frac{n_i(n_i+1)}{2}-1.\]
\end{theorem}

For a flat torus $(T^n,g_0)$, the holonomy is trivial, so $\dim\varepsilon(g)=\frac{n(n+1)}{2}-1$. On a flat manifold $(M,g)$ of any dimension, the moduli space of flat metrics consists precisely of those $g+h$, $h\in\varepsilon(g)$, which are positive-definite. In particular, the moduli space of flat metrics forms a connected component in $\EMod$, and all IED are integrable. In dimension $n\leq4$, it is known that Einstein metrics on $T^n$ are automatically flat, so in these cases $\EMod$ is connected
\cite[12.17]{besse}.

In the (irreducible or flat) examples above, we see that all IED are integrable since they come from infinitesimal deformations of the $G$-structure. In \cite{holrig} it is shown that the IED of a Riemannian product of Ricci-flat manifolds are integrable if this holds for the IED of the factors. Together with the structure theorem for Ricci-flat manifolds we can finally conclude:

\begin{theorem}
If $(M,g)$ admits a parallel spinor, then all IED of $g$ are integrable.
\end{theorem}

As a consequence of Theorem~\ref{thm:stablefromintegrable}, we thus obtain a stronger stability result.

\begin{corollary}
If $(M,g)$ admits a parallel spinor, then $g$ is $\EH$-stable.
\end{corollary}

In fact, this result can be strengthened in two ways. First, the precise statement in \cite{holrig} says that \emph{every} metric of nonnegative scalar curvature near $g$ has to admit a parallel spinor. Second, in some cases one can show that $g$ is supreme (cf.~\S\ref{sec:yamabe}), \cite{LeB99}.

\begin{theorem}
Let $(M^n,g)$ be compact, simply connected, and irreducible.
\begin{enumerate}[\upshape(i)]
    \item If $(M,g)$ is Calabi--Yau, then $g$ is supreme if and only if $n\not\equiv6\mod 8$.
    \item If $(M,g)$ is a torsion-free $\rmG_2$-manifold, then $g$ is not supreme.
    \item If $(M,g)$ is torsion-free $\Spin(7)$-manifold, then $g$ is supreme.
\end{enumerate}
\end{theorem}

As a consequence of Theorem~\ref{thm:yamabepos}, a Ricci-flat metric on $M$ is supreme if and only if $M$ does not admit a metric of positive scalar curvature. There are obstructions to the existence of such metrics on spin manifolds, such as the index of the Dirac operator $D$ in dimension $4m$, or the Hitchin invariant $\dim\ker D\mod2$ in dimension $8m+2$. LeBrun \cite{LeB99} shows that in the Calabi--Yau ($n\not\equiv6\mod 8$) or $\Spin(7)$ case, one of these invariants does not vanish, which rules out the existence of a positive scalar curvature metric. On the other hand, the \emph{Atiyah invariant}, whose vanishing is an equivalent condition for the existence of a positive scalar curvature metric on a compact, simply connected spin manifold, automatically vanishes in dimensions $8m+6$ and $7$ \cite{LM,Stolz}.


Lastly, we mention a classical example in dimension four where the Einstein moduli space is understood. As a consequence of the Hitchin--Thorpe inequality, every Einstein metric on the $K3$ manifold or any of its quotients is hyper-Kähler. Algebraic geometry allows us to identify $\EMod$ with a dense open subset of the locally symmetric space
\[(\Orth(3,19)\cap\GL(22,\Z))\backslash \Orth(3,19) / (\Orth(3)\times\Orth(19))\]
of dimension $57$, see for example \cite[\S12.K]{besse}.

\section{Kähler--Einstein metrics}
\label{sec:ke}

Many of the methods of complex and spin geometry that are applicable for Calabi--Yau manifolds still work for the more general Kähler--Einstein metrics. Since the Ricci-flat case was already treated in the preceding section, we shall focus on Kähler--Einstein metrics with nonzero scalar curvature.

\subsection{Stability}

First, there is a spin$^\rmc$-analogue of Theorem~\ref{thm:parallelsemistable} by Dai--Wang--Wei \cite{DWW07}.

\begin{theorem}
\label{thm:spincsemistable}
Let $(M,g)$ be a compact, spin$^\rmc$, Einstein manifold of nonpositive scalar curvature. If $(M,g)$ admits a parallel spin$^c$-spinor, then $g$ is semistable.
\end{theorem}

For the rest of this section, we let $(M,g,J)$ be a compact Kähler--Einstein manifold with $\Ric_g=Eg$. Then $M$ admits a compatible spin$^\rmc$ structure with a parallel spin$^c$-spinor. In fact, simply connected, irreducible spin$^\rmc$-manifolds carrying a parallel spin$^c$-spinor are known to either be Kähler or admit a parallel spinor \cite{spinc}.

\begin{corollary}
\label{thm:KEsemistable}
If $g$ is Kähler--Einstein with $E<0$, then $g$ is semistable.
\end{corollary}

It is, however, not clear whether they are also $\EH$-stable. In dimension four, LeBrun used Seiberg--Witten theory to show that the answer is yes \cite[Thm.~3, Thm.~4]{LeB95}, cf.~\cite[Thm.~3.6]{LeB99}:

\begin{theorem}
\label{thm:KEsupreme}
If $g$ is Kähler--Einstein with $E<0$ and $\dim M=4$, then $g$ is supreme. Conversely, every other supreme Einstein metric on $M$ is Kähler--Einstein.
\end{theorem}

In contrast, Petean showed that any compact simply connected manifold of dimension $\geq5$ has nonnegative Yamabe invariant, so negative scalar curvature metrics on it are never supreme \cite{petean}.

Turning to positive scalar curvature, there is a simple criterion for instability \cite{CHI04}.

\begin{theorem}
\label{thm:KEunstable}
If $g$ is Kähler--Einstein with $E>0$, then $\coind(g)\geq b_2(M)-1$.
\end{theorem}

The only known examples where $\coind(g)>b_2(M)-1$ are symmetric spaces; see (7) and (9) in Theorem~\ref{thm:symmstable}. The proof idea for Theorem~\ref{thm:KEunstable} is to produce destabilizing directions directly from harmonic forms. A recent result by Biquard--Ozuch \cite{cKunstable} on Einstein metrics which are conformally Kähler, particular to dimension four, uses a similar idea.

\begin{theorem}
If $(M^4,g)$ is Einstein with $E>0$, conformally Kähler, and not half-conformally flat, then $\coind(g)\geq b_2^-(M)$.
\end{theorem}

\subsection{Infinitesimal Einstein and complex deformations}
\label{sec:einsteincomplex}

Analogous to the decomposition of differential forms into $(p,q)$-forms, symmetric $2$-tensors split according to how they interact with the complex structure $J$. In the sequel, we freely identify $2$-tensors with endomorphisms of $TM$ using the metric $g$,
\[h(X,Y)=g(hX,Y)\qquad\forall X,Y\in\X(M).\]

\begin{definition}
A $2$-tensor $h\in\Gamma(T^\ast M\otimes T^\ast M)$ is called \emph{(anti-)Hermitian} if it (anti-)\allowbreak commutes with $J$. The bundles of Hermitian/anti-Hermitian symmetric $2$-tensors are respectively denoted with
\begin{align*}
\Sym^+T^\ast M&=\{h\in\Sym^2T^\ast M\,|\,h\circ J=J\circ h\},\\
\Sym^-T^\ast M&=\{h\in\Sym^2T^\ast M\,|\,h\circ J=-J\circ h\}.
\end{align*}
\end{definition}

Since $J$ is parallel, these subbundles are also parallel. Their spaces of sections shall be denoted with $\Sy^\pm(M)$. Koiso \cite[Prop.~7.3]{Koiso83} has shown that the space of IED splits accordingly,
\[\varepsilon(g)=\varepsilon^+(g)\oplus\varepsilon^-(g),\qquad \varepsilon^\pm(g)=\varepsilon(g)\cap\Sy^\pm(M).\]
Moreover, there is a parallel bundle isomorphism between Hermitian symmetric $2$-tensors and real $(1,1)$-forms,
\[\Lambda^{1,1}_\R T^\ast M\to\Sym^+T^\ast M,\qquad \alpha\mapsto\alpha\circ J.\]
This isomorphism maps primitive $(1,1)$-forms (the orthogonal complement of the Kähler form) to traceless symmetric $2$-tensors, and coclosed $(1,1)$-forms to divergence-free $2$-tensors. In particular, since $\LL$ commutes with parallel bundle maps, primitive harmonic $(1,1)$-forms correspond to tt-tensors in the kernel of $\LL$. Together with the fact that harmonic $2$-forms on Kähler--Einstein manifolds with $E>0$ are always of type $(1,1)$, this proves Theorem~\ref{thm:KEunstable}.

On the other hand, since the Hodge--de Rham Laplacian is nonnegative, we obtain a vanishing result.
\begin{theorem}
\label{thm:hermitian}
\begin{enumerate}[\upshape(i)]
\item If $E<0$, then $\varepsilon^+(g)=0$.
\item If $E=0$, then $\dim\varepsilon^+(g)=h^{1,1}-1$.
\end{enumerate}
\end{theorem}

Infinitesimal Einstein deformations are closely linked to infinitesimal deformations of the complex structure. This was thoroughly investigated by Koiso in \cite{Koiso83}. Recall that a complex structure on a manifold $M$ is a section $J\in\Gamma(\End TM)$ such that $J^2=-\Id$ whose Nijenhuis tensor vanishes, $N_J=0$. Differentiating this in $J$, while staying orthogonal to the $\Diff$-orbit of $J$, which has tangent space
\[T_J(\Diff\cdot J)=\Lie_{\X(M)}J,\]
leads to the following definition.

\begin{definition}
Let $(M,g,J)$ be a complex manifold with a Riemannian metric. An \emph{infinitesimal complex deformation} (ICD) of $J$ is a section $I\in\Gamma(\End TM)$ such that
\begin{enumerate}[\upshape(i)]
    \item $IJ+JI=0$, i.e.~$I$ is anti-Hermitian,
    \item $N_J'(I)=0$, and
    \item $I\perp\Lie_{\X(M)}J.$
\end{enumerate}
\end{definition}

Koiso characterized ICD in terms of the Dolbeault operator for the twisted complex
\[0\to\Gamma(T^{1,0}M)\stackrel{\bar\partial^\nabla}{\to}\Omega^{0,1}(M,T^{1,0}M)\stackrel{\bar\partial^\nabla}{\to}\Omega^{0,2}(M,T^{1,0}M)\to\ldots\]
The cohomology of this complex is, by Dolbeault's theorem, isomorphic to the sheaf cohomology of the sheaf $\Theta$ of holomorphic vector fields.

\begin{theorem}
\label{thm:ICDdolbeault}
Let $(M,g,J)$ be a Kähler manifold.
\begin{enumerate}[\upshape(i)]
    \item An anti-Hermitian section $I\in\Gamma(\End TM)$, viewed as an element of the space $\Omega^{0,1}(M,T^{1,0}M)$, is an ICD if and only if $\bar\partial^\nabla I=0$ and $(\bar\partial^\nabla)^\ast I=0$.
    \item The space of ICD of $J$ is canonically isomorphic to $H^1(M,\Theta)$.
\end{enumerate}
\end{theorem}

This establishes the connection with the well-known deformation theory of Kodaira--Spencer \cite{KNS}.

We return to the setting of Kähler--Einstein manifolds. Using the metric, one may view an ICD also as a contravariant $2$-tensor. Koiso showed that the symmetric and skew-symmetric parts of an ICD are again ICD.

\begin{theorem}
\label{thm:ICDsymm}
Let $(M,g,J)$ be Kähler--Einstein.
\begin{enumerate}[\upshape(i)]
    \item If $I\in\Gamma(T^\ast M\otimes T^\ast M)$ is an ICD, then $\LL I=2EI.$
    \item The space of symmetric ICD coincides with $\varepsilon^-(g)$.
    \item Skew-symmetric ICD coincide with the parallel $2$-forms of type $(2,0)+(0,2)$. If $E=0$, these are precisely the harmonic forms of this type; if $E\neq0$, all ICD are symmetric.
\end{enumerate}
\end{theorem}

A consequence is the dimension formula from the previous section (Theorem~\ref{thm:calabiyau}). Together with Theorem~\ref{thm:hermitian} we also obtain the following.

\begin{corollary}
\label{thm:KEdimIED}
Suppose $E\neq0$. Then
\[\dim\varepsilon(g)\geq 2\dim_\C H^1(M,\Theta),\]
with equality if $E<0$.
\end{corollary}

\subsection{Integrability}

One may compare the deformation theory of Einstein metrics with that of complex structures \cite{KNS}. The obstructions to the integrability of ICD (cf.~Section~\ref{sec:deformations}) live in the second cohomology $H^2(M,\Theta)$. In particular, if $H^2(M,\Theta)=0$, then every infinitesimal complex deformation is integrable into an analytic curve of complex structures. See \cite{DWW07} for some instructive examples of $H^1(M,\Theta)$ and $H^2(M,\Theta)$ on complex manifolds.

In the Kähler--Einstein case, the deformation theories of Einstein metrics and of complex structures interact in subtle ways. We summarize the results of \cite{Koiso83}.

\begin{theorem}
Let $(M,g,J)$ be Kähler--Einstein.
\begin{enumerate}[\upshape(i)]
    \item 
    Suppose $E<0$ and all ICD are integrable. Then
    then so are all IED; moreover, every Einstein metric near $g$ is also Kähler (for some complex structure), and the Kähler--Einstein metrics form an open subset of $\EMod$.
    \item 
    Suppose $E>0$, all ICD are integrable, and $(M,J)$ has no holomorphic vector fields. Then all IED in $\varepsilon^-(g)$ are integrable, and the Kähler--Einstein metrics form a submanifold of $\widetilde\EMod_g$ with tangent space $\varepsilon^-(g)$.
\end{enumerate}
\end{theorem}

However, even if $E<0$, in which case there is a one-to-one correspondence between IED and ICD, it could be the case that there exist integrable IED whose corresponding ICD are not integrable.

One may ask whether the Kähler--Einstein metrics always form an open subset in the Einstein moduli space, that is, whether any nearby Einstein metric is also Kähler. As in Section~\ref{sec:parallelspinor}, we call this \emph{holonomy rigidity}. This is true in dimension four, but in contrast to the Calabi--Yau case, it remains unanswered in general.

\begin{theorem}
\label{thm:KE4hol}
If $\dim M=4$, any Kähler--Einstein metric (with $E\neq0$) on $M$ is holonomy rigid.
\end{theorem}
\begin{proof}
Let $g$ be Kähler--Einstein with Einstein constant $E$. For $E<0$, the statement follows from Theorem~\ref{thm:KEsupreme}: the metric $g$ is supreme, and since $\EH$ is locally constant on $\EMod$ (Theorem~\ref{thm:EMod}), any nearby Einstein metric is also supreme, hence Kähler--Einstein.

On the other hand, if $(M^4,g,J)$ is Kähler--Einstein with $E>0$, then $(M,J)$ is a del Pezzo surface \cite{CLW}. By \cite[Cor.~1]{LeBrunAGAG}, $g$ is holonomy rigid.
\end{proof}

\begin{questype}{Open question}
Are Kähler--Einstein metrics (with $E\neq0$) holonomy rigid?
\end{questype}

Without any assumption on the ICD, using Theorem~\ref{thm:hermitian} and an explicit formula for the second order obstruction, Nagy--Semmelmann \cite{NS25} have shown that IED are at least integrable to second order:

\begin{theorem}
\label{thm:ICD2}
Let $(M,g,J)$ be Kähler--Einstein with $E<0$. Then all IED are integrable to second order.
\end{theorem}

Dai--Wang--Wei \cite{DWW07} have further asked whether there exist compact Kähler--Einstein manifolds with $E<0$ with nonintegrable ICD or IED. In dimension four, Horikawa showed the existence of examples where ICD are obstructed to second order \cite[\S6]{Hori}. As a consequence of Theorem~\ref{thm:KE4hol}, the corresponding IED are also obstructed (although necessarily to higher order by Theorem~\ref{thm:ICD2}). In view of Theorem~\ref{thm:KEsupreme}, this example also shows that the converse of Theorem~\ref{thm:stablefromintegrable} is false. In dimensions other than four, the question by Dai--Wang--Wei is still open.

\begin{questype}{Open question}
Does there exist a compact Kähler--Einstein manifold $(M,g,J)$ of negative scalar curvature with $\dim M\neq4$ admitting nonintegrable ICD or IED?
\end{questype}

In the positive scalar curvature case, IED seem to be generically obstructed; see Section~\ref{sec:symmetric}, in particular Theorem~\ref{thm:symm2ndorder} (5).


\section{Quaternion-Kähler metrics}
\label{sec:qk}

Quaternion-Kähler manifolds are Riemannian manifolds $(M^{4m}, g)$ with holonomy group contained in $\Sp(m)\cdot \Sp(1) \subset \SO(4m)$.
For $m \ge 2$ they are automatically Einstein; see \cite[\S14]{besse} for an introduction. Furthermore, they are de Rham irreducible if the scalar curvature is diﬀerent from zero. If the scalar curvature vanishes, the holonomy reduces further to $\Sp(m)$, i.e., the manifold is then hyper-Kähler and falls under Section~\ref{sec:parallelspinor}. In the case of negative scalar curvature, Kröncke--Semmelmann proved the following result \cite{KS24}.

\begin{theorem}
\label{thm:qknegstable}
Every (compact or noncompact) quaternion-Kähler manifold of negative scalar curvature is strictly stable.
\end{theorem}

In Salamon's $\qE$-$\qH$-formalism, where $\qH$ and $\qE$ are the locally defined vector bundles associated to the standard representations of $\Sp(1)$ resp.~$\Sp(n)$, so that $TM^\C\cong\qH\otimes\qE$, the bundle of symmetric $2$-tensors splits as
\[\Sym^2T^\ast M^\C\cong(\Sym^2\qH\otimes\Sym^2\qE)\oplus\Lambda^2_0\qE\oplus\C.\]
The proof of Theorem~\ref{thm:qknegstable} is based on the simple observation that there is an injective parallel bundle map $\Sym^2T^\ast M\to\Lambda^4T^\ast M$, which takes the Lichnerowicz Laplacian $\LL$ on symmetric $2$-tensors to the nonnegative Hodge Laplacian on $4$-forms. Hence $\LL\geq0$ on $\Sy^2(M)$, and the metric is strictly stable because $E<0$ by assumption.

For \emph{positive} quaternion-Kähler manifolds ($E>0$), the situation is more complicated and the stability question is still open. The only known complete positive examples are symmetric spaces, called \emph{Wolf spaces}. They are all strictly stable with the exception of the complex $2$-plane Grassmannian $\Gr_2\C^n$, $n\geq4$, which is semistable but $\EH$-unstable, cf.~Theorem~\ref{thm:symmstable} and Corollary~\ref{thm:symmunstable}.

\begin{questype}{LeBrun--Salamon Conjecture}
Every (compact) positive quaternion-Kähler manifold is isometric to a Wolf space.
\end{questype}

On a positive quaternion-Kähler manifold, the Lichnerowicz Laplacian on $\Sy^2_0(M)$ satisfies an a-priori bound which is not quite enough for stability \cite{H06}.

\begin{theorem}
If $(M^{4m},g)$ is a positive quaternion-Kähler manifold, then $\LL\geq2E\frac{m+1}{m+2}$ on trace-free symmetric $2$-tensors.
\end{theorem}

In the following unpublished result of Homma--Semmelmann \cite{HS25}, the possible destabilizing directions are narrowed down.

\begin{theorem}
Let $(M^{4m}, g)$ be a positive quaternion-Kähler manifold. The eigenspace $\ker(\LL-\lambda)|_{\TT(M)}$ to an eigenvalue $2E\frac{m+1}{m+2}\leq\lambda\leq 2E$ is
\begin{enumerate}[\upshape(i)]
    \item isomorphic to $\ker(\Delta-\lambda)|_{C^\infty(M)}$ if $2E\frac{m+1}{m+2}<\lambda\leq2E$,
    \item isomorphic to $\ker(\LL-2E\frac{m+1}{m+2})|_{\Gamma(VM)}$ if $\lambda=2E\frac{m+1}{m+2}$, where $VM$ is such that $VM^\C=\Sym^2\qH\otimes\Sym^2\qE$.
\end{enumerate}
\end{theorem}

By a result of LeBrun \cite{LeB88}, a positive quaternion-Kähler metric admits no nontrivial deformation through quaternion-Kähler metrics. It is however unknown whether it can have Einstein deformations of generic holonomy. A resolution of the LeBrun--Salamon Conjecture would reduce this problem to the study of $\Gr_2\C^n$, the only Wolf space admitting IED (see Theorem~\ref{thm:symm2ndorder} for partial progress on the latter).

\section{Manifolds with Killing spinors}
\label{sec:killingspinor}

Let $(M^n,g)$ be a Riemannian spin manifold with spinor bundle $\Sigma M$. A \emph{Killing spinor} is a section $\varphi\in\Gamma(\Sigma M)$ such that $\nabla_X\varphi = \lambda X \cdot \varphi$ holds for all vector fields $X$. Here $\nabla$ is the connection on  $\Sigma M$ induced from the Levi-Civita connection of $g$, $\cdot$ is the Clifford multiplication and $\lambda$ is some complex constant.

If such a Killing spinor exists, the metric $g$ is Einstein and the constant $\lambda$ is related to the scalar curvature by $E = 4(n-1) \lambda^2$. Hence the constant $\lambda$ is either real or purely imaginary. A nonzero Killing spinor is called \emph{real} in the first and \emph{imaginary} in the second case. If $\lambda=0$, the manifold is Ricci-flat and we are in the case of Section~\ref{sec:parallelspinor}. So we assume $\lambda\neq0$ in this section. For more details on spin geometry and Killing spinors see \cite{BFGK}.

Baum \cite{Baum} showed that manifolds with imaginary Killing spinors can be written as warped products.

\begin{theorem}
Any complete, connected spin manifold carrying a imaginary Killing spinor is isometric to $(\R \times M, dr^2 + e^{2 r}g)$, where $(M,g)$ is a complete, connected spin manifold with a parallel spinor.
\end{theorem}

As a consequence of Theorems~\ref{thm:kr1} and~\ref{thm:parallelsemistable}, we obtain stability, cf.~\cite{W}.

\begin{corollary}
Complete manifolds with imaginary Killing spinors are strictly stable.
\end{corollary}

Connected spin manifolds with real Killing spinors with Killing number $\lambda \neq 0$ are irreducible and either non-symmetric or of constant curvature. Since they have $E>0$, they are compact if they are complete; thus, in the sequel we assume that $M$ is compact.

On the standard sphere, the spinor bundle is trivialized by Killing spinors. Moreover, the metric cone over a manifold with a Killing spinor admits a parallel spinor. This allows a partial classification of manifolds with real Killing spinors, due to Bär \cite{Baer}, which is listed in Table~\ref{tab:killingspinors}.


\begin{table}[ht]
    \centering
    \renewcommand{\arraystretch}{1.2}
    \begin{tabular}{|c|c|c|c|}\hline
    $\Hol(M^n,\nabla^{\mathrm{c}})\subseteq$&Description&Dim.~KS&Metric cone\\\hline\hline
    trivial&standard sphere $(S^n,g_0)$&$2^{\lfloor n/2\rfloor}$&Euclidean space\\\hline
    $\SU(3)$, $n=6$&strict nearly Kähler&$1$&torsion-free $\rmG_2$\\\hline
    $\rmG_2$, $n=7$&proper nearly parallel $\rmG_2$&$1$&torsion-free $\Spin(7)$\\\hline
    $\U(m)$,&\multirow{2}*{Sasaki--Einstein}&\multirow{2}*{$2$}&\multirow{2}*{Calabi--Yau}\\
    $n=2m+1$&&&\\\hline
    $\Sp(m)\cdot\Sp(1)$,&\multirow{2}*{3-Sasaki}&\multirow{2}*{$m+2$}&\multirow{2}*{hyper-Kähler}\\
    $n=4m+3$&&&\\\hline
    \end{tabular}
    \caption{Classes of manifolds with real Killing spinors. Here $\nabla^{\mathrm{c}}$ denotes the \emph{canonical connection} of $(M,g)$, a spin connection with parallel skew-symmetric torsion which parallelizes the Killing spinors (KS).}
    \label{tab:killingspinors}
\end{table}

The standard sphere is well known to be strictly stable. For all other classes we have partial results stating instability under certain additional conditions or for special cases.

A weak stability result was shown by Gibbons--Hartnoll--Pope \cite{GHP03}, but also follows more generally from Theorem~\ref{thm:kr2} and the semistability of the metric cone, Theorem~\ref{thm:parallelsemistable}.

\begin{theorem}
\label{thm:killingphys}
Spin manifolds admitting a real Killing spinor are physically stable.
\end{theorem}

\begin{questype}{Open question}
Do there exist stable manifolds with
real Killing spinors other than the
standard sphere?
\end{questype}

\subsection{Nearly Kähler manifolds in dimension $6$}

A nearly Kähler manifold is an almost Hermitian manifold $(M, J, g)$ with $(\nabla_XJ)X=0$ for all tangent vectors $X$, where $\nabla$ is the Levi-Civita connection of $g$. It is called \emph{strict} if it is not Kähler. The dimension $6$ turns out to be special: 
only in this dimension strict nearly K\"ahler manifolds admit Killing spinors and hence are Einstein.
 Compact, six-dimensional strict nearly Kähler manifolds are also called \emph{Gray manifolds}. The homogeneous Gray manifolds have been classified by Butruille \cite{But}: one has
\[S^6 = \rmG_2/\SU(3),\quad S^3 \times S^3 = \frac{\SU(2)^3}{\diag(\SU(2))}, \quad \C\Proj^3 = \SO(5)/\U(2), \quad F_{1,2} = \SU(3)/T^2.\]
In all cases the nearly Kähler metric is the normal metric induced by a multiple of the Killing form (see Section~\ref{sec:normal}). It is the round metric on $S^6$, but not the product metric on $S^3\times S^3$ and not the Kähler--Einstein metric in the last two cases.

In addition there exist inhomogeneous nearly Kähler metrics of cohomogeneity one on $S^6$ and $S^3\times S^3$, constructed by Foscolo--Haskins \cite{FH}. The full classification of Gray manifolds is an open problem.

Semmelmann--Wang--Wang \cite{SWW20} gave a topogical lower bound on the Hilbert coindex by constructing eigentensors of $\LL$ from harmonic forms.

\begin{theorem}
\label{thm:NK}
Let $(M^6,g, J)$ be a Gray manifold. Then $\coind(g)\geq b_2(M)+b_3(M)$.
\end{theorem}

In particular, all known examples of Gray manifolds except $S^6$ are unstable, and a simply connected \emph{semistable} Gray manifold must be a rational homology sphere. Schwahn \cite{Sch22b} showed that this bound is attained for the homogeneous examples, while Solé-Farré \cite{cohom1coindex} improved the lower bound to $\coind(g)\geq4$ for the Foscolo--Haskins metrics on $S^3\times S^3$.

The proof of Theorem  \ref{thm:NK} uses the same idea as that of Theorem \ref{thm:KEsemistable}. A harmonic $2$-form on a strict nearly Kähler manifold has to be a section of $\Lambda^{1,1}_\R T^\ast M$, which is isomorphic to $\Sym^+T^\ast M$. Similarly, a harmonic $3$-form must be a section of $\Lambda^{(2,1)+(1,2)}_\R T^\ast M\cong\Sym^-T^\ast M$. The involved bundle isomorphisms are not parallel under $\nabla$, but under the canonical connection $\nabla^{\mathrm{c}}$. These isomorphisms interchange the Hodge Laplacian with the standard Laplacian of $\nabla^{\mathrm{c}}$. A comparison formula between $\LL$ and this standard Laplacian finally shows that a harmonic form is mapped to a destabilizing $\LL$-eigentensor.

The space of IED of Gray manifolds was studied by Moroianu--Semmelmann \cite{MS11}.

\begin{theorem}
Let $(M^6,g)$ be a Gray manifold. Then
\[\varepsilon(g)\cong\ker(\Delta-2E/5)\oplus\ker(\Delta-6E/5)\oplus\ker(\Delta-12E/5)\]
where $\Delta$ is considered on primitive co-closed real $(1,1)$-forms. The last eigenspace corresponds to the space of infinitesimal deformations of the nearly Kähler structure.
\end{theorem}

Of the homogeneous examples, only the flag manifold $\SU(3)/T^2$ admits IED \cite{MS10}, which are at the same time infinitesimal nearly Kähler deformations. Foscolo \cite{F17} developed the deformation theory for nearly Kähler structures and showed that the infinitesimal deformations on $\SU(3)/T^2$ are not integrable, i.e.~the metric corresponds to an isolated point in the moduli space of nearly Kähler structures. In fact, the metric is also isolated in the Einstein moduli space, see Theorem~\ref{thm:flagrigid}.

\subsection{Nearly parallel $\rmG_2$-manifolds}
\label{sec:nearlyg2}

A nearly parallel $\rmG_2$-manifold is a complete $7$-dimensional spin manifold $(M^7,g)$ with a real Killing spinor. Equivalently, one has a $3$-form $\sigma\in\Omega^3(M)$ whose stabilizer at each point is isomorphic to the group $\rmG_2$ and such that $d\sigma = \tau_0 \star \sigma$ for some $\tau_0\in\mathbb{R}\setminus\{0\}$. 
This determines the scalar curvature by $E=3\tau_0^2/8$. 
Note that $\tau_0\neq0$ excludes the torsion-free case, which was already discussed in \ref{sec:parallelspinor}.
Excluding the standard sphere, these manifolds fall into three classes, depending on whether the dimension of the space of real Killing spinors is 3, 2 or 1. These are, respectively, $3$-Sasaki manifolds, Sasaki--Einstein manifolds which are not $3$-Sasaki, and \emph{proper} nearly parallel $\rmG_2$ manifolds. We shall mostly focus on the proper case; Sasaki--Einstein and $3$-Sasaki manifolds will be considered separately, in general dimension, in Section~\ref{sec:SE}.

Homogeneous nearly parallel $\rmG_2$ manifolds have been classified by Friedrich et al.~\cite{FKMS}. In the proper case we have the squashed $7$-sphere $(S^7, g_{\mathrm{sq}})$, the isotropy-irreducible Berger space $\SO(5)/\SO(3)_{\mathrm{irr}}$ and the Aloff--Wallach spaces $N_{k,l}=\SU(3)/\U(1)_{k,l}$, where the embedding depends on nonzero $k,l\in\Z$ with $(k,l) \neq (1,1)$.

Similar to Theorem~\ref{thm:NK}, Semmelmann--Wang--Wang gave a topological bound for the Hilbert coindex of nearly parallel $\rmG_2$-manifolds \cite{SWW22}.

\begin{theorem}\label{thm:g2-A}
Let  $(M^7,g)$ be a nearly parallel $\rmG_2$-manifold. Then $\coind(g)\geq b_3(M)$.
\end{theorem}

The proof is similar to that of Theorem~\ref{thm:NK}; it uses that harmonic $3$-forms are sections of 
the subbundle $\Lambda^3_{27}T^\ast M\subset\Lambda^3T^\ast M$, which is defined by the irreducible $27$-dimensional representation of $\rmG_2$, and which is the image of $\Sym^2_0T^\ast M$ under $\nabla^{\mathrm{c}}$-parallel bundle isomorphism $h\mapsto h\diamond\sigma$.

This bound is not sharp, as evidenced by the homogeneous examples: the squashed sphere $(S^7,g_{\mathrm{sq}})$ is linearly unstable as a consequence of Theorem~\ref{thm:submersion1}, the Aloff--Wallach spaces have $b_3=0$ and are linearly unstable by the work of Wang--Wang \cite{WW2} (cf.~Section~\ref{sec:nonnormal}), and on the Berger space $\SO(5)/\SO(3)_{\mathrm{irr}}$, which is a rational homology sphere, one can find destabilizing directions by Fourier-analytic methods (\cite{SWW22}, see Section~\ref{sec:normal}).

It is know that $3$-Sasaki manifolds have $b_3=0$.
There are several examples of complete Sasaki--Einstein manifolds with $b_3>0$. However we do not know any example of a proper nearly parallel $\rmG_2$-manifold with nonzero third Betti number.

The space of IED of a nearly parallel $\rmG_2$-manifold has been studied by Alexandrov--Semmelmann \cite{nearlyg2}.

\begin{theorem}
Let $(M^7,g,\sigma)$ be a nearly parallel $\rmG_2$-manifold with $d\sigma = \tau_0 \star \sigma$. Then
\[\varepsilon(g)\cong\ker(\star d+\tau_0)|_{\Omega^3_{27}(M)}\oplus\ker(\star d-\tau_0/2)|_{\Omega^3_{27}(M)}\oplus\ker(dd^\ast-\tau_0^2/2)|_{\Omega^3_{27}(M)}\]
where $\Omega^3_{27}(M)=\Gamma(\Lambda^3_{27}T^\ast M)$.
\end{theorem}

The first summand corresponds to infinitesimal deformations of the nearly parallel $\rmG_2$-structure. It is further shown in \cite{nearlyg2} that the Berger space and the squashed $S^7$ do not admit IED, while the proper nearly parallel $\rmG_2$-metric on $N_{1,1}$ does, and all its IED are also infinitesimal nearly parallel $\rmG_2$-deformations. Later, Nagy--Semmelmann \cite{NS21,NS23} showed that these do not integrate to an actual family of nearly parallel $\rmG_2$-metrics.

Besides the homogeneous examples, there is one other known source of proper nearly parallel $\rmG_2$-manifolds. This is the second, \emph{squashed} Einstein metric $g_{\mathrm{sq}}$ in the canonical variation (see Section~\ref{sec:subm}) on a $3$-Sasaki manifold $(M^7,g)$. In fact, these manifolds fiber over $4$-dimensional orbifolds. The proper nearly parallel $\rmG_2$-metric $g_{\mathrm{sq}}$ is obtained by scaling the $3$-Sasaki metric on the vertical distribution $\mathcal{V}$, spanned by the three Reeb vector fields $\xi_1,\xi_2,\xi_3\in\X(M)$, by a factor of $1/5$. Examples are the squashed Einstein metric on $S^7$ (fibering over $S^4$) and one of the homogeneous Einstein metrics on $N_{1,1}$ (fibering over $\C\Proj^2$). A complete description of the space of destabilizing directions in terms of eigenfunctions was given by Nagy--Semmelmann \cite[Thm.~1.3]{NS24}. Let $C^\infty_{\mathrm{b}}(M)$ denote the space of \emph{basic functions}, that is,
\[C^\infty_{\mathrm{b}}(M)=\{f\in C^\infty(M)\,|\,\Lie_{\xi_1}f=\Lie_{\xi_2}f=\Lie_{\xi_3}f=0\}.\]

\begin{theorem}\label{thm:g2-B}
Let $(M^7, g)$ be a $3$-Sasaki manifold, $\Ric_g=6g$, and let $g_{\mathrm{sq}}$ be the nearly parallel $\rmG_2$-metric in the canonical variation of $g$, for which $\tau_0=12/\sqrt{5}$.
\begin{enumerate}[\upshape(i)]
    \item The space of IED of $g_{\mathrm{sq}}$ is isomorphic to
    \[\varepsilon(g_{\mathrm{sq}})\cong\ker(\star d+\tau_0)\big|_{\Omega^3_{27}(M)}\cong\ker(\Delta-24)|_{C^\infty_{\mathrm{b}}(M)}.\]
    \item Suppose that $g$ does not have constant sectional curvature. Then the sum of the destabilizing eigenspaces of $\LL|_{\TT(M)}$ for the metric $g_{\mathrm{sq}}$ is isomorphic to
    \[\R \oplus \mathbf{H}^-_4 \oplus \bigoplus_{16 < \mu < 24} \ker (\Delta - \mu)|_{C^\infty_{\mathrm{b}}(M)}.\]
    The summand $\R$ is spanned by the destabilizing direction from Theorem~\ref{thm:submersion1}, while the space $\mathbf{H}^-_4$ corresponds to certain horizontal harmonic $2$-forms, cf.~\cite[(41)]{NS23}.
\end{enumerate}
\end{theorem}

In case of the squashed metric $g_{\mathrm{sq}}$ on $N(1,1)$, the only destabilizing direction comes from the canonical variation, so $\coind(g_{\mathrm{sq}})=1$ \cite[Thm.~9.3]{NS24}. There are no known examples of squashed metrics with $\coind(g_{\mathrm{sq}})>1$.

In summary we can say that all known examples of proper nearly parallel $\rmG_2$-manifolds are linearly unstable.

\subsection{Sasaki--Einstein and $3$-Sasaki manifolds}
\label{sec:SE}


A similar result to Theorems~\ref{thm:NK} and~\ref{thm:g2-A} also holds in the Sasaki--Einstein case \cite{SWW22}.

\begin{theorem}\label{thm:ES2}
Let $(M^{2m+1}, g)$ be a Sasaki--Einstein manifold, $m\geq2$. Then $\coind(g)\geq b_2(M)$. Moreover, if $m=3$, then $\coind(g)\geq b_2(M)+b_3(M)$.
\end{theorem}

That is, if a simply connected $7$-dimensional Sasaki--Einstein manifold is semistable, it must be a rational homology sphere.

A Sasaki--Einstein manifold $(M^{2m+1},g)$ is called \emph{regular} if the Reeb vector field generates a free circle action. This implies that $(M,g)$ is a principal circle bundle over a Kähler--Einstein manifold $(B^{2m},g_B)$ of positive scalar curvature, and the fibration is a Riemannian submersion with totally geodesic fibers. C.~Wang showed that if $(B,g_B)$ has a tt-eigentensor to a small enough eigenvalue, then it can be pulled back to a destabilizing direction on $(M,g)$ \cite{W}.

\begin{theorem}
\label{thm:sasakiregular}
Let $(M^{2m+1},g)$ be a regular Sasaki--Einstein manifold, $\Ric_g=2mg$, and $\pi: (M,g)\to (B,g_B)$ the corresponding circle bundle, $\Ric_{g_B}=(2m+2)g_B$. If there exists $h\in\TT(B)$ such that $\Ltwoinprod{\LL h}{h}<(4m-4)\Ltwoinprod{h}{h}$, then $\pi^\ast h\in\TT(M)$ is a destabilizing direction for $(M,g)$.
\end{theorem}

For a regular Sasaki--Einstein manifold $(M,g)$ fibering over an irreducible Hermitian symmetric space, instability does not follow from Theorem~\ref{thm:ES2}, as it has $b_2(M)=0$. However, Wang--Wang~\cite{WW2} proved the linear instability of the Sasaki--Einstein metric on the Stiefel manifold $\SO(n+2)/\SO(n)$, $n\geq3$, which fibers over the complex quadric $\SO(n+2)/(\SO(n)\times\SO(2))$, and of $\Sp(n)/\SU(n)$ for $n\ge 4$, fibering over the Hermitian symmetric space $\Sp(n)/\U(n)$ (see Section~\ref{sec:nonnormal}).

By a result of Pedersen--Poon \cite{PP} it is known that $3$-Sasaki metrics cannot be deformed through a nontrivial curve of $3$-Sasaki metrics. It is unknown whether a metric with a Killing spinor can have an Einstein deformation to a metric without a Killing spinor. At least, the analogous result to the holonomy rigidity for parallel spinors \cite{holrig} is false for Killing spinors and their canonical connection: as shown by van Coevering \cite{Coe17}, there exist examples of $7$-dimensional toric $3$-Sasaki metrics, admitting deformations through Sasaki--Einstein metrics
with just two Killing spinors.

\section{Homogeneous spaces}
\label{sec:homogeneous}

An excellent source of examples are manifolds which are homogeneous. Throughout this final section we assume that $M^n$ is homogeneous and fix a presentation $M=G/H$, where $G$ is a compact, connected Lie group and $H$ a closed subgroup (the isotropy subgroup of a given \emph{base point} $o\in M$). Without restriction we also assume that the action $G\curvearrowright M$ is \emph{almost effective}, that is, only a discrete subgroup of $G$ acts trivially on $M$.

\subsection{Invariant Einstein metrics}

Denote the Lie algebras of $G$ and $H$ by $\gf$ and $\hf$, respectively, and choose a reductive decomposition
\[\gf=\hf\oplus\mf,\]
that is, $\Ad(H)\mf\subset\mf$. Then the $H$-module $\mf$ is canonically identified with the isotropy representation $T_oM$, and $G$-invariant tensor fields on $M$ correspond to $H$-invariant tensors on $\mf$. In particular, we can identify
\[\Sy^2(M)^G\cong(\Sym^2\mf^\ast)^H,\]
thus the open cone $\Met^G\subset\Sy^2(M)^G$ of $G$-invariant metrics is a manifold of dimension at most $n(n+1)/2$. A natural gauge group acting on $\Met^G$ is the group
\[\Aut(G/H)=\{\varphi\in\Aut(G)\,|\,\varphi(H)=H\},\]
which acts on $M$ by diffeomorphisms and on $\Met^G$ by pullback. We note that two $G$-invariant metrics might be isometric via a diffeomorphism not in $\Aut(G/H)$. Deciding this is a difficult problem; see \cite[Thm.~5]{DZ79} for left-invariant metrics on compact simple Lie groups. At least on the infinitesimal level it is true that
\begin{equation}
T_g(\Aut(G/H)\cdot g)=T_g(\Diff\cdot g\cap\Met^G)
\label{tgdiffinvariant}
\end{equation}
for any $g\in\Met^G$.

\begin{definition}
\begin{enumerate}[\upshape(i)]
    \item The moduli space of invariant Einstein metrics is defined as
    \[\EMod^G=\{g\in\Met_1^G\,|\, g\text{ is Einstein}\}/\Aut(G/H).\]
    \item A $G$-invariant Einstein metric $g$ is called \emph{$G$-rigid} if its class $[g]$ is isolated in $\EMod^G$.
\end{enumerate}
\end{definition}


There are other sensible choices of gauge groups. First, one might take only the inner automorphisms contained in $\Aut(G/H)$, which leads us to consider the normalizer $N_G(H)$, resp.~the subgroup $N_G(H)/Z(G)\subset\Aut(G/H)$. Second, one could take the group $\Diff^G$ of diffeomorphisms commuting with the $G$-action. By virtue of $G$ being compact, this group is isomorphic to $N_G(H)/H$ \cite[\S1, Cor.~4.3]{Bredon}, albeit acting on $M$ in a different way than $N_G(H)$. However, both $N_G(H)$ and $\Diff^G$ still have the same orbits on $\Met^G$. Further, $N_G(H)/Z(G)$ and $\Aut(G/H)$ share the same identity component, so the notion of $G$-rigidity stays the same even if one considers invariant Einstein metrics only up to the action of $N_G(H)$.

The moduli space $\EMod^G$ has been studied for example by Böhm--Wang--Ziller \cite{BWZ04}. They show that the set of Einstein metrics in $\Met_1^G$ is a semialgebraic set. In particular, it admits a local stratification into real algebraic smooth submanifolds, and is locally arcwise connected. A consequence of the local arcwise connectedness of both $\EMod^G$ and $\EMod$ together with \eqref{tgdiffinvariant} is that rigidity (in the sense of Section~\ref{sec:moduli}) implies $G$-rigidity.

Recall that homogeneous Einstein metrics necessarily have $E\geq0$, with $E=0$ implying that the manifold is a torus (see Section~\ref{sec:parallelspinor} for a description of the Einstein moduli space), while $E>0$ implies that $G/H$ has finite fundamental group. In the latter case, the semisimple part of $G$ still acts transitively on $G/H$, so we may assume without restriction that $G$ is semisimple. Then $G$ has only finitely many outer automorphisms; in particular, $\Aut(G/H)$ is compact.

A key result by Böhm--Wang--Ziller \cite{BWZ04} is the following.

\begin{theorem}
\label{thm:invariantmodulicompact}
If $G/H$ has finite fundamental group, then $\EMod^G$ is compact.
\end{theorem}

Looking at the known examples in fact motivates the following conjecture of theirs.

\begin{questype}{Finiteness Conjecture}
If $G/H$ has finite fundamental group, then $\EMod^G$ is finite. Equivalently, all $G$-invariant Einstein metrics are $G$-rigid.
\end{questype}

Invariant Einstein metrics may also be characterized purely variationally.

\begin{theorem}
\label{thm:critinvariant}
A metric $g\in\Met_1^G$ is Einstein if and only if it is a critical point of $\EH|_{\Met_1^G}$.
\end{theorem}

In fact, for an appropriate notion of volume fixing, this remains true for non-compact unimodular $G$ \cite[Lem.~3.6]{Gstab1}. The discussion in Section~\ref{sec:stable} behaves well under restriction to $G$-invariant variations, as worked out by J.~Lauret \cite{Gstab1}. For example,
\begin{equation}
T_g\Met_1^G=\Sy^2_0(M)^G=T_g(\Aut(G/H)\cdot g)\oplus\TT(M)^G
\label{decompinvariant}
\end{equation}
for any $g\in\Met_1^G$. If we denote $\mf^\hf=\{X\in\mf\,|\,[\hf,X]=0\}$, then $\hf\oplus\mf^\hf$ is the Lie algebra of $N_G(H)$ and $\Aut(G/H)$. The following theorem tells us that the space of trivial variations is in many cases not a problem.

\begin{theorem}
For any $X\in\mf$, let $S(X)\in\Sym^2\mf^\ast$ be given by
\[S(X)(Y,Z)=g_o(\pr_\mf[X,Y],Z)+g_o(\pr_\mf[X,Z],Y),\qquad Y,Z\in\mf.\]
Then, at any $g\in\Met^G$, the space $T_g(\Aut(G/H)\cdot g)$ is identified with
\[S(\mf^\hf)\subset(\Sym^2_0\mf^\ast)^H.\]
Moreover, $T_g(\Aut(G/H)\cdot g)=0$ if any one of the following conditions is satisfied:
\begin{enumerate}[\upshape(i)]
    \item $\mf^\hf=0$,
    \item $G/H$ is naturally reductive, i.e.~$S=0$,
    \item $\mf$ is multiplicity-free, e.g.~if $\rk(G)=\rk(H)$.
\end{enumerate}
\end{theorem}


The variational characterization in Theorem~\ref{thm:critinvariant} and the splitting~\eqref{decompinvariant} motivate the following $G$-invariant version of stability.

\begin{definition}[$G$-stability]
A $G$-invariant Einstein metric $g$ is called
\begin{enumerate}[\upshape(i)]
    \item strictly (linearly) $G$-stable if $\EH_g''$ is negative-definite on $\TT(M)^G$,
    \item $G$-stable if $g$ is a local maximum of $\EH$ restricted to $\Met_1^G$,
    \item (linearly) $G$-semistable if $\EH_g''$ is negative-semidefinite on $\TT(M)^G$,
    \item $G$-unstable if it is not $G$-stable,
    \item linearly $G$-unstable if there exists $h\in\TT(M)^G$ such that $\EH_g''(h,h)>0$,
    \item strongly $G$-unstable if $\EH_g''$ is positive-definite on $\TT(M)^G$.
\end{enumerate}
\end{definition}

As before, the implications (i)$\Rightarrow$(ii)$\Rightarrow$(iii) and (vi)$\Rightarrow$(v)$\Rightarrow$(iv) hold. The dimension of the maximal subspace of $\TT(M)^G$ on which $\EH_g''$ is positive-definite is called the \emph{$G$-coindex}, and is denoted by $\coind_G(g)$. In contrast to Definition~\ref{def:stable}, even though the Hilbert coindex is always finite, strong $G$-instability is a possibility since $\TT(M)^G$ is also finite-dimensional.

Of course, an invariant Einstein metric being stable in the sense of Definition~\ref{def:stable} implies the corresponding notion of $G$-stability.

The variational characterization of homogeneous Einstein metrics (Theorem~\ref{thm:critinvariant}) can sometimes also be used to predict their existence. For example, Wang--Ziller proved the following result \cite{WZ86}.

\begin{theorem}
$\EH|_{\Met_1^G}$ is bounded from above and proper if and only if $\hf\subset\gf$ is a maximal subalgebra. In this case, $\EH|_{\Met_1^G}$ assumes its global maximum, so there is at least one $G$-stable invariant Einstein metric.
\end{theorem}

In most other cases, we expect $G$-instability. Under certain conditions, the \emph{simplicial complex theorem} of Böhm \cite{Boe04}, a generalization of the \emph{graph theorem} \cite{BWZ04}, guarantees the existence of saddle points of $\EH|_{\Met_1^G}$, that is, $G$-unstable homogeneous Einstein metrics.

\subsection{Harmonic analysis}
\label{sec:harmana}

A great advantage of the homogeneous setting is that analytic problems on homogeneous vector bundles can often be reduced to questions of representation theory. Given a unitary representation $\pi: H\to\GL(V)$ of the isotropy group $H$, the \emph{induced representation} of $G$ is the space $L^2(G,V)^H$ of $V$-valued functions on $G$ which are equivariant under the right-action of $H$, that is
\[L^2(G,V)^H=\{f\in L^2(G,V)\,|\,f(xh)=\pi(h)^{-1}f(x)\ \forall x\in G,\ h\in H\},\]
with $G$ acting by left-translation, i.e.
\[(x\cdot f)(y)=f(x^{-1}y),\qquad x,y\in G,\ f\in L^2(G,V)^H.\]
This unitary representation of $G$ is canonically isomorphic to the \emph{left-regular representation} $L^2(M,VM)$ of $L^2$-sections of the homogeneous vector bundle $VM=G\times_\pi V$, with $G$ acting via pullback.

We let $\widehat G$ denote the set of equivalence classes of irreducible representations of $G$ (over $\C$ usually parametrized by their highest weights). A consequence of the Peter--Weyl theorem and Frobenius reciprocity is the following decomposition of $G$-modules.

\begin{theorem}
\label{thm:peterweyl}
If $VM=G\times_\pi V$ is a homogeneous vector bundle, then
\[L^2(M,VM)\cong\closedsum_{\gamma\in\hat G}[\gamma]\otimes\Hom_H([\gamma],V).\]
\end{theorem}

The finite-dimensional vector space $\Hom_H([\gamma],V)$ simply counts the multiplicity of the representation $[\gamma]$ in $L^2(M,VM)$. Its elements are called the \emph{Fourier coefficients} for the \emph{Fourier mode} $\gamma$.

As a consequence of Schur's Lemma, any $G$-equivariant (differential) operator between homogeneous vector bundles $D:\Gamma(VM)\to \Gamma(WM)$ descends to a discrete family $(D|_\gamma)_{\gamma\in\widehat G}$ of linear maps
\[D|_\gamma: \Hom_H([\gamma],V)\longrightarrow\Hom_H([\gamma],W).\]
To the study of stability of homogeneous Einstein metrics, we are of course interested in the case where $D=\LL: \Sy^2_0(M)\to\Sy^2_0(M)$. Since $\delta$ is also $G$-equivariant, restricting $\LL$ to tt-tensors amounts to considering the restriction of $\LL|_\gamma$ to suitable subspaces of $\Hom_H([\gamma],\Sym^2_0\mf)$, cf.~Corollary~\ref{thm:ttLL}.

With the reductive decomposition $\gf=\hf\oplus\mf$ fixed, the $H$-principal bundle $G\to G/H$ admits a distinguished invariant connection, the \emph{canonical reductive connection}, defined by projecting the Maurer--Cartan form of $G$ to $\hf$. From this we obtain a covariant derivative $\bar\nabla$ on any homogeneous vector bundle with the property that $G$-invariant sections are $\bar\nabla$-parallel. Differentiating a section of $VM$ with respect to $\bar\nabla$ thus corresponds to taking a directional derivative in $C^\infty(G,V)^H$. This makes expressing invariant differential operators in terms of $\bar\nabla$ convenient for computational purposes, as we will see later.

\subsection{Symmetric spaces}
\label{sec:symmetric}

The linear stability and infinitesimal deformability of the symmetric Einstein metric on simply-connected, irreducible symmetric spaces of compact type was studied by Koiso \cite{Koiso80} and decided for most cases, with the last gaps filled by Gasqui--Goldschmidt \cite{GG96}, Schwahn \cite{Sch22a} and Semmelmann--Weingart \cite{SW22}.

\begin{theorem}
\label{thm:symmstable}
The simply connected, irreducible symmetric spaces of compact type fall into three classes:
\begin{enumerate}[\upshape(i)]
    \item those which are semistable and admit IED:
    \begin{enumerate}[\upshape(1)]
        \item $\SU(n)/\SO(n)$ with $n\geq 3$,
        \item $\SU(n)=(\SU(n)\times\SU(n))/\diag(\SU(n))$ with $n\geq3$,
        \item $\SU(2n)/\Sp(n)$ with $n\geq3$,
        \item $\rmE_6/\rmF_4$,
        \item $\Gr_k\C^{k+l}=\SU(k+l)/\rmS(\U(k)\times\U(l))$ with $k,l\geq2$,
    \end{enumerate}
    \item those which are linearly unstable and do not admit IED:
    \begin{enumerate}[\upshape(1)]
    \setcounter{enumii}{5}
        \item $\Sp(n)$ with $n\geq2$,
        \item $\Sp(n)/\U(n)$ with $n\geq3$,
        \item $\Sp(k+l)/(\Sp(k)\times\Sp(l))$ with $k,l\geq2$,
        \item $\SO(5)/(\SO(2)\times\SO(3))$,
    \end{enumerate}
    \item all others, which are strictly stable.
\end{enumerate}
\end{theorem}

Koiso utilized the fact that for these metrics, the Lichnerowicz Laplacian $\LL$ coincides with the quadratic Casimir operator of $G$ on the left-regular representation, see Theorem~\ref{thm:laplacenormal}.

\begin{definition}
Let $Q$ be a $G$-invariant inner product on $\gf$, and $\pi: G\to\GL(W)$ a representation of $G$. The \emph{quadratic Casimir operator} of $G$ on $W$ is the $G$-equivariant endomorphism defined by
\[\Cas^{G,Q}_W=-\sum_i\pi_\ast(X_i)^2\in\End_G(W),\]
where $(X_i)$ a $Q$-orthonormal basis of $\gf$.
\end{definition}

Since $\Cas^{G,Q}_W$ is $G$-equivariant, it must act as multiplication by a constant when $W$ is irreducible. In the case where $G$ is semisimple, irreducible complex representations are parametrized by their highest weights, and this constant can be easily calculated using Freudenthal's formula.

\begin{theorem}
\label{thm:freudenthal}
Let $G$ be a compact semisimple Lie group, $Q$ a $G$-invariant inner product on $\gf$, and assume that a maximal torus $\tf$ and a set of simple roots have been chosen. Then, for any dominant integral weight $\gamma$, the Casimir operator acts on the highest weight module $[\gamma]$ as multiplication by the constant
\[\Cas^{G,Q}_{\gamma}=\langle\gamma,\gamma+2\delta_\gf\rangle,\]
where the inner product on $\tf^\ast$ is induced by $Q$, and $\delta_\gf=\frac12\sum_{\alpha>0}\alpha$ is the half-sum of positive roots of $\gf$.
\end{theorem}

Together with Theorem~\ref{thm:peterweyl} and Corollary~\ref{thm:ttLL}, this allows one to find the spectrum of $\LL$ on tt-tensors and thus prove Theorem~\ref{thm:symmstable}. Using the same approach for the Laplacian on functions and the shrinker entropy $\nu$, Cao--He analyzed the dynamical stability of compact symmetric spaces under Ricci flow \cite{CH15}.

For the spaces $G/H$ in Theorem~\ref{thm:symmstable} (i), it remains to check whether they are $\EH$-stable, and whether their IED are integrable. In each of the cases (1)--(5) it turns out that $\varepsilon(g)\cong\gf$ as a $G$-module -- that is, the space of IED is abstractly isomorphic to the space of Killing fields. This isomorphism was made explicit in \cite{BHMW,NS25,HSS25a,HSS25b}.

The spaces (1)--(4) may be described as the compact duals of the spaces
\[\SL(n,\mathbb{A})/\SU(n,\mathbb{A})\]
where $\mathbb{A}$ runs through the normed division algebras $\R,\C,\Ham,\Oct$, and $n=3$ in case (4). They are closely related to simple formally real Jordan algebras, and the bundle map from Killing fields to IED can be thought of as coming from the Jordan multiplication \cite{HSS25b}. In contrast, on the complex Grassmannians (5) the parametrization works via a certain first order differential operator \cite{HSS25a}.

In all cases (1)--(5), the integrability of IED to second order was studied using the expression in Theorem~\ref{thm:secondobst}. For (1), (3) and (5), the second order obstruction $\Psi$ is a nonzero multiple of the unique invariant cubic form on $\gf=\su(m)$, resp.~in case (2) of a particular invariant cubic form on $\gf=\su(n)\oplus\su(n)$. On the exceptional space (4), the second order obstruction vanishes identically as there are no invariant cubic forms on $\ef_6$. These results allow one to explicitly describe the locus of IED which are integrable to second order \cite{BHMW,NS25,HSS25a,HSS25b}.

\begin{theorem}
\label{thm:symm2ndorder}
Let $(G/H,g)$ be one of the spaces in Theorem~\ref{thm:symmstable} (i).
\begin{enumerate}[\upshape(i)]
    \item The set $\mathcal{Q}\subset\varepsilon(g)$ of IED which are integrable to second order is
    \begin{itemize}
        \item isomorphic to the variety
        \[\left\{X\in\su(m)\,\middle|\,X^2=\frac{\tr(X^2)}{m}\Id_m\right\}\]
        in cases (1), (3) and (5) ($\gf=\su(m)$),
        \item isomorphic to its diagonal image in $\gf=\su(m)\oplus\su(m)$ in case (2),
        \item equal to $\varepsilon(g)\cong\gf=\ef_6$ in case (4).
    \end{itemize}
    \item In particular, in cases (1), (2), (5) with $m$ odd, $\mathcal{Q}=0$ and hence the symmetric Einstein metric is rigid.
\end{enumerate}
\end{theorem}

In \cite{Jen71,Nik21,BWZ04} it was proven, by considering nearby metrics invariant under a smaller transitive group, that $\SU(n)$ (2) and $\SU(2n)/\Sp(n)$ (3) are (nonlinearly) $\EH$-unstable despite being semistable. With Theorem~\ref{thm:symm2ndorder} and Corollary~\ref{thm:nonlinearunstable}, we have a more general result.

\begin{corollary}
\label{thm:symmunstable}
The Einstein manifolds (1), (2), (3) and (5) from Theorem~\ref{thm:symmstable} are $\EH$-unstable.
\end{corollary}

The $\EH$-stability of $\rmE_6/\rmF_4$ (4) as well as the integrability to third order on all of the spaces (1)--(5) is currently unclear, but we strongly suspect the symmetric Einstein metrics to be rigid. We remark that the first ever example of an Einstein metric which has IED, but which is rigid, namely the symmetric Einstein metric on $\C\Proj^{2n}\times\C\Proj^1$, was given by Koiso \cite{Koiso82}; Theorem~\ref{thm:symm2ndorder} gives the first irreducible examples.

There is an analogous, but as of today incomplete discussion on the solitonic rigidity of compact symmetric spaces and their dynamical stability under Ricci flow. The complex projective spaces and complex Grassmannians, which are semistable for the shrinker entropy, admit infinitesimal solitonic deformations (ISD) from $2E$-eigenfunctions \cite{CH15}. Using Theorem~\ref{thm:cubicobst}, Knopf--Šešum \cite{KS19} and Kröncke \cite{Kr20} independently showed the dynamical instability of $\C\Proj^n$, $n\geq2$, even though it is strictly stable for $\EH$. Similarly, Hall--Murphy--Waldron showed the dynamical instability of $\Gr_k\C^n$ for $n\neq 2k$ \cite{HMW21}.

Those ISD of $\C\Proj^n$ which are integrable to second order have been completely characterized \cite{Kr16}. As it turns out, they are all obstructed to second order if $n$ is even. Remarkably, Li--Zhang \cite{LZ23} have worked out the third order obstruction for conformal ISD (cf.~Theorem~\ref{thm:nu2nd} (ii))
for odd $n$, showing that none of them are integrable.

\begin{theorem}
The Fubini--Study metric $\C\Proj^n$ is solitonically rigid, that is, every Ricci soliton close to it must be isometric to it (up to scale).
\end{theorem}

\subsection{Normal homogeneous spaces}
\label{sec:normal}

The method of reducing calculations to Casimir operators works more generally on normal homogeneous spaces, that is, homogeneous spaces $M=G/H$ with a reductive decomposition $\gf=\hf\oplus\mf$ which is orthogonal under some $G$-invariant inner product $Q$ on $\gf$, and with an invariant Riemannian metric $g$ determined by $g_o=Q|_\mf$. If $G$ is semisimple and $Q=-B_\gf$ is minus the Killing form, then $g$ is called the \emph{standard metric}.

Let $VM=G\times_\pi V$ be a homogeneous vector bundle, let $\bar\nabla$ denote the canonical reductive connection from Section~\ref{sec:harmana}, $\bar R$ its curvature on $TM$, and
\[\bar\Delta=\bar\nabla^\ast\bar\nabla+\K(\bar R,VM)\]
be the \emph{standard Laplacian} on $\Gamma(VM)$ associated to $\bar\nabla$ \cite{SW19}. The following theorem generalizes the fact that the Lichnerowicz Laplacian coincides with a Casimir operator on symmetric spaces \cite{MS10}.

\begin{theorem}
\label{thm:laplacenormal}
$\bar\Delta=\Cas^{G,Q}_{\Gamma(VM)}$ and $\K(\bar R,VM)=\Cas^{H,Q}_V$.
\end{theorem}

Indeed, on Riemannian symmetric spaces, the Levi-Civita connection $\nabla$ is the same as $\bar\nabla$, so $\LL=\bar\Delta$. For general naturally reductive homogeneous spaces (which includes the normal ones), the two connections differ by a torsion term:
\begin{equation}
\nabla=\bar\nabla+\frac12\A,
\label{conndiff}
\end{equation}
where $\A$ is the $G$-invariant $(2,1)$-tensor given at the basepoint by
\[\A_XY=\pr_\mf[X,Y],\qquad X,Y\in\mf.\]
Naturally reductive spaces are geometries with parallel skew-symmetric torsion -- indeed, the torsion of $\bar\nabla$ is exactly $-\A$. The presence of these torsion terms complicates the calculation of eigenvalues of $\LL$ on all non-trivial tensor bundles.  However, using \eqref{conndiff} a formula for $\LL$ on symmetric $k$-tensors, involving various Casimir operators, was worked out by Schwahn \cite{Sch24}.

\begin{theorem}
\label{thm:LLnormal}
Consider the inclusion $\Sym^k\mf\subset\Sym^k\gf$ and the $H$-equivariant projection $\pr_{\Sym^k\mf}: \Sym^k\gf\to\Sym^k\mf$, as well as the inclusion
\[\Sy^k(M)\cong C^\infty(G,\Sym^k\mf)^H\subset C^\infty(G,\Sym^k\gf)\cong C^\infty(G)\otimes\Sym^k\gf.\]
If $g$ is a normal Einstein metric on $G/H$ induced by the invariant inner product $Q$ on $\gf$, with Einstein constant $E$, then
\begin{align*}
\LL=\,&\frac32\Cas^{G,Q}_{\Sy^k(M)}+\pr_{\Sym^k\mf}\left(\Cas^{G,Q}_{\Sym^k\gf}-\frac12\Cas^{G,Q}_{C^\infty(G)\otimes\Sym^k\gf}\right)\\
&-\frac32\Cas^{H,Q}_{\Sym^k\mf}-kE+\frac{k}{4}.
\end{align*}
\end{theorem}

Restricted to $G$-invariant tensors $\Sy^k(M)^G$, i.e.~the trivial Fourier mode $0\in\widehat G$, the Lichnerowicz Laplacian reduces to
\[\LL|_0=\frac12\A^\ast\A=-\frac12\sum_i\A_{e_i}^2,\]
cf.~\cite{LW22,Gstab1}. This has been used by J.~Lauret \cite{Gstab1} to analyze the linear $G$-stability and compute the $G$-coindex of the normal Einstein metrics on the families
\[\SU(kn)/\rmS(\U(n)^k),\qquad \Sp(kn)/\Sp(n)^k,\qquad \SO(kn)/\rmS(\Orth(n)^k),\qquad (k\geq3)\]
as well as the Jensen metrics on $(H\times K)/\diag(K)$, where $H$ is simple and $K\subset H$ is semisimple. Suppose $\kf=\kf_1\oplus\ldots\oplus\kf_r$ is a $B_\hf$-orthogonal decomposition into simple ideals of the Lie algebra $\kf$ of $K$, and that $B_{\kf_i}=cB_\hf|_{\kf_i}$ for some common constant $c>0$. For most of these cases, he shows that the Jensen metric is strongly $G$-unstable with $\coind_G(g)=r$.

Normal homogeneous Einstein manifolds $G/H$, at least with $G$ simple, were classified by Wolf~\cite{Wolf} (for the isotropy irreducible case, cf.~Krämer~\cite{Kraemer} and Manturov~\cite{Man1,Man2,Man3}) and Wang--Ziller \cite{WZ85} (for the isotropy reducible case).

Lauret--Lauret \cite{LL23} systematically analyzed the linear $G$-stability of these isotropy reducible spaces and explicitly computed the $G$-coindex for most cases. An interesting case is the space $\SO(2n)/T^n$, $n\geq3$, which admits $G$-invariant IED and is $G$-semistable for $n\geq4$. However, it is (nonlinearly) $G$-unstable, as shown by Lauret--Will \cite{Gstab3}.

Notable is also the following criterion of Nikonorov \cite[Thm.~2]{Nik21}.

\begin{theorem}
Suppose that $G$ is semisimple and $G/H$ is isotropy reducible. If the standard metric on $G/H$ satisfies $E>2/5$, then it is strongly $G$-unstable.
\end{theorem}

Schwahn--Semmelmann--Weingart \cite{SSW23} showed that the normal Einstein metric on $\rmE_7/\PSO(8)$ is strictly stable (in the general sense), giving the first example of a non-symmetric stable Einstein metric with $E>0$. Later, using Theorem~\ref{thm:LLnormal}, Schwahn \cite{Sch24} analyzed the linear stability of the spaces in \cite{Wolf,WZ85}, producing more examples of both strictly and semistable positive Einstein metrics which are not symmetric.

\begin{theorem}
\label{thm:stablenormal}
There exist non-symmetric normal homogeneous Einstein manifolds (with $E>0$) which are strictly stable.
\end{theorem}

We note that the results are not exhaustive; due to the method providing only lower bounds on the eigenvalues of $\LL|_{\TT(M)}$, there are still several candidates of normal homogeneous spaces left where the stability type is unclear. For the same list of spaces, Lauret--Tolcachier \cite{LT25} carried out a similar analysis on functions, allowing them to conclude dynamical stability for all the strictly stable examples found in \cite{Sch24}.

An interesting isotropy-irreducible example is the Berger space $\SO(5)/\SO(3)_{\mathrm{irr}}$, defined by the irreducible five-dimensional representation of $\SO(3)$. Its normal metric is nearly parallel $\rmG_2$, but its instability does not follow from Theorem~\ref{thm:g2-A} as it has $b_3=0$. However, Semmelmann--Wang--Wang \cite{SWW22} observed that for representation-theoretic reasons, there exist trace-free Killing $2$-tensors on $\SO(5)/\SO(3)_{\mathrm{irr}}$ which are destabilizing eigentensors of $\LL$ -- they satisfy equality in~\eqref{estimate}, and the nearly parallel $\rmG_2$-structure allows calculating the eigenvalues of $\K(R)$.

For non-simple $G$, homogeneous spaces $G/H$ on which the standard metric is Einstein are not fully classified, but there is a partial classification due to Nikonorov \cite{N00}. Gutiérrez--Lauret showed that many of the spaces appearing in this classification scheme are linearly $G$-unstable \cite{GL23}.

Turning to the deformability of non-symmetric homogeneous Einstein metrics, the only known examples which admit non-invariant IED are the nearly Kähler flag manifold $\SU(3)/T^2$ \cite{MS10} and the Aloff--Wallach space $N_{1,1}=(\SU(3)\times\SO(3))/\U(2)\cong\SU(3)/\U(1)$, a proper nearly parallel $\rmG_2$-manifold \cite{nearlyg2}. In both cases, $\varepsilon(g)\cong\su(3)$ as a $G$-module. By rewriting the second order obstruction with \eqref{conndiff} and making the torsion terms explicit, Schwahn \cite{Sch22b} showed the following.

\begin{theorem}
\label{thm:flagrigid}
The normal Einstein metric on $\SU(3)/T^2$ is rigid.
\end{theorem}

Combined with our current knowledge about symmetric spaces, we are tempted to make the following conjecture, which is a strengthening of the Finiteness Conjecture.

\begin{questype}{Conjecture}
Homogeneous Einstein metrics with $E>0$ are rigid.
\end{questype}

\subsection{Non-normal metrics}
\label{sec:nonnormal}

Many known homogeneous Einstein metrics are not normal or even naturally reductive, rendering the calculation of the spectrum of $\LL$ difficult. At least on $G$-invariant tensors, the groundwork laid in \cite{Gstab1} makes some calculations possible.

Generalizing the work in \cite{Gstab1}, Lauret--Will \cite{Gstab2} gave an explicit formula for $\LL|_0$ on symmetric $2$-tensors for \emph{any} $G$-invariant metric in terms of the structural constants of $G/H$, where $G$ is compact and semisimple. Let again $Q$ denote a $G$-invariant inner product on $\gf$; then, for any invariant metric $g$, there exists a $Q$-orthogonal decomposition
\[\mf=\mf_1\oplus\ldots\oplus\mf_r\]
into irreducible $H$-modules such that
\[g_o=x_1Q|_{\mf_1}+\ldots+x_rQ|_{\mf_r}\]
for some numbers $x_i>0$. If $(e_\alpha^{(i)})$ denotes a $Q$-orthonormal basis of $\mf_i$, then the \emph{structural constants} are defined by
\[[ijk]=\sum_{\alpha,\beta,\gamma}Q([e_\alpha^{(i)},e_\beta^{(j)}],e_\gamma^{(k)})^2,\qquad 1\leq i,j,k\leq r.\]
\begin{theorem}
\label{thm:LLinvariant}
If $\mf$ is multiplicity-free, there is a basis of $(\Sym^2\mf^\ast)^H\cong\R^r$ for which the matrix of $\LL|_0$ has the entries
\begin{align*}
(\LL|_0)_{kk}&=\frac{1}{d_k}\sum_{i,j\neq k}\frac{x_k}{x_ix_j}[ijk]+\frac{1}{d_k}\sum_{i\neq k}\frac{x_i}{x_k^2}[ikk],\\
(\LL|_0)_{kl}&=\frac{1}{\sqrt{d_kd_l}}\sum_i\frac{x_i^2-x_k^2-x_l^2}{x_ix_kx_l}[ikl],\qquad k\neq l,
\end{align*}
where $d_k=\dim\mf_k$.
\end{theorem}

We close this chapter by giving an overview of miscellaneous stability results for homogeneous Einstein metrics. Using Theorem~\ref{thm:LLinvariant}, Lauret--Will \cite{Gstab2} analyzed the linear $G$-stability of the homogeneous Einstein metrics on most of the generalized Wallach spaces, the exceptional flag manifolds  with $b_2=1$, and many other examples. In light of Theorem~\ref{thm:KEunstable}, these flag manifolds are particularly interesting; it turns out that all their invariant Einstein metrics are linearly $G$-unstable, except for the unique invariant Kähler--Einstein metric, which is strictly $G$-stable. This makes them candidates for stability in the non-invariant sense.

\begin{questype}{Open question}
Except for symmetric spaces, are the invariant Kähler--Einstein metrics on flag manifolds with $b_2=1$ stable?
\end{questype}

Let $H$ be simple, and $K\subset H$ be a closed subgroup of positive dimension. Invariant Einstein metrics on spaces of the form $(H\times H)/\diag(K)$, and their $G$-stability, were studied by Lauret--Will \cite{LW25} and Gutiérrez \cite{G24} -- these Einstein metrics are never normal. In the case where $H/K$ is an irreducible symmetric space with $K$ simple, the invariant Einstein metrics are completely classified and are all linearly $G$-unstable.

Wang--Wang \cite{WW2} have studied the stability of low-dimensional homogeneous Einstein manifolds. In particular, they show linear $G$-instability for all invariant Einstein metrics on the Aloff--Wallach spaces $N_{k,l}$, and for the invariant Sasaki--Einstein metric on the Stiefel manifold $\SO(n+2)/\SO(n)$, $n\geq3$. They also note that the Lichnerowicz eigenvalue of the destabilizing tt-tensors on $\Sp(n)/\U(n)$, $n\geq4$ (cf.~Theorem~\ref{thm:symmstable}) is small enough to apply Theorem~\ref{thm:sasakiregular} and obtain linear instability of the regular Sasaki--Einstein manifold $\Sp(n)/\SU(n)$, $n\geq4$.

Up to dimension $7$, compact homogeneous Einstein metrics are almost completely classified, except on $S^3\times S^3$ (see \cite{s3s3}). Combined with \cite{SWW20,SWW22}, we can state the following results.

\begin{theorem}
All non-symmetric, compact, simply connected homogeneous Einstein manifolds of dimension $5\leq n\leq 7$ are linearly unstable, except possibly the unknown Einstein metrics on $S^3\times S^3$.
\end{theorem}

\begin{theorem}
The regular Sasaki--Einstein manifolds fibering over compact Hermitian symmetric spaces are dynamically unstable under the Ricci flow.
\end{theorem}

Having discussed the known results on homogeneous Einstein manifolds, we note that symmetries are still useful even if they do not act transitively. For example, in the cohomogeneity one setting, the Einstein equations (and their linearization) can be turned into ODE problems. This has enabled several recent constructions of inhomogeneous Einstein metrics. A notable example are the Böhm metrics on spheres and products of spheres, constructed in \cite{Boe98}. Through a combination of analytic and numerical methods, Gibbons--Hartnoll--Pope \cite{GHP03} have shown that some of these Einstein metrics are physically unstable -- their results even suggest that among the Böhm metrics, the lowest eigenvalue of $\LL$ on tt-tensors can become arbitrarily negative.

Overall, stable Einstein metrics with $E>0$ appear to be quite rare -- in fact, the only currently known examples are symmetric spaces and the ones from Theorem~\ref{thm:stablenormal}.

\section*{Acknowledgments}
\addcontentsline{toc}{section}{Acknowledgments}

P.~Schwahn was supported by FAPESP project no.~\textbf{2024/08127-4}, part of the BRIDGES collaboration. U.~Semmelmann acknowledges support by the DFG priority program \textbf{SPP 2026 ``Geometry at Infinity''}.

An abridged version of this survey will appear in the final volume reporting on the SPP 2026. We are grateful for this opportunity.

The authors would like to thank Christoph Böhm, Hans-Joachim Hein, Klaus Kröncke and Claude LeBrun for helpful and inspiring discussions, and Klaus Kröncke and Matthias Ludewig for their comments on an earlier version of the manuscript.


\begin{thebibliography}{199.}%
\addcontentsline{toc}{section}{References}\bibitem{nearlyg2} B.~Alexandrov, U.~Semmelmann: \emph{Deformations of nearly parallel $\rmG_2$-structures}, Asian J. Math. 16 (4), 713--744 (2012)
\bibitem{holrig} B.~Ammann, K.~Kröncke, H.~Weiss, F.~Witt: \emph{Holonomy rigidity for Ricci-flat metrics}, Math. Z. 291, 303--311 (2019)
\bibitem{anderson} M.~T.~Anderson: \emph{A survey of Einstein metrics on 4-manifolds}, \url{https://arxiv.org/abs/0810.4830}
\bibitem{AM} L. Andersson, V. Moncrief: \emph{Einstein spaces as attractors for the Einstein flow}, J. Differential Geom. 89, no. 1, 1--47 (2011)
\bibitem{Baer} C.~Bär: \emph{Real Killing spinors and holonomy}, Comm. Math. Phys. 154, no. 3, 509--521 (1993)
\bibitem{BHM17} W.~Batat, S.~J.~Hall, T.~Murphy: \emph{Destabilising compact warped product Einstein manifolds}, Comm. Anal. Geom. 29 (5), 1061--1094 (2021)
\bibitem{BHMW} W.~Batat, S.~J.~Hall, T.~Murphy, J.~Waldron: \emph{Rigidity of $\SU_n$-type symmetric spaces}, Internat. Math. Research Notices 2024 (3), 2066--2098 (2024)
\bibitem{Baum} H.~Baum: \emph{Complete Riemannian manifolds with imaginary Killing spinors}, Ann. Glob. Anal. Geom. 7, no. 3, 205--226 (1989)
\bibitem{BFGK} H.~Baum, T.~Friedrich, R.~Grunewald, I.~Kath: \emph{Twistors and Killing spinors on Riemannian manifolds},  Teubner-Texte zur Mathematik, 124. B.~G.~Teubner Verlagsgesellschaft mbH, Stuttgart (1991)
\bibitem{s3s3} F.~Belgun, V.~Cortés, A.~S.~Haupt, D.~Lindemann: \emph{Left-invariant Einstein metrics on $S^3\times S^3$}, J. Geom. Phys. 128, 128--139 (2018)
\bibitem{besse} A.~L.~Besse: \emph{Einstein manifolds}, Ergebnisse der Mathematik und ihrer Grenz\-ge\-bie\-te 3rd Series 10, Springer-Verlag, Berlin (1987)
\bibitem{BiqI} O.~Biquard: \emph{Désingularisation de métriques d’Einstein I}, Invent. Math. 192 (1), 197–252 (2013)
\bibitem{BiqII} O.~Biquard: \emph{Désingularisation de métriques d’Einstein II}, Invent. Math. 204 (2), 473–504 (2016)
\bibitem{cKunstable} O.~Biquard, T.~Ozuch: \emph{Instability of conformally Kähler, Einstein metrics}, \url{https://arxiv.org/abs/2310.10109}, to appear in J. Diff. Geom. (2023)
\bibitem{Boe98} C.~Böhm: \emph{Inhomogeneous Einstein metrics on low-dimensional spheres and other low-dimensional spaces}, Invent. math. 134, 145-176 (1998)
\bibitem{Boe04} C.~Böhm: \emph{Homogeneous Einstein metrics and simplicial complexes}, J. Diff. Geom. 67, 79--165 (2004)
\bibitem{BWZ04} C.~Böhm, M.~Wang, W.~Ziller: \emph{A variational approach for compact homogeneous Einstein manifolds}, Geom. Funct. Anal. 14, 681--733 (2004)
\bibitem{Boe05} C.~Böhm: \emph{Unstable Einstein metrics}, Math. Z. 250, 279--286 (2005)
\bibitem{BGK05} C.~P.~Boyer, K.~Galicki, J.~Kollár: \emph{Einstein metrics on spheres}, Ann. Math. 162 (1), 557--580 (2005)
\bibitem{BK25} L.~Branca, K.~Kröncke: \emph{Some isolation and stability results for Einstein manifolds}, \url{https://arxiv.org/abs/2506.13248} (2025)
\bibitem{Bredon} G.~Bredon, \emph{Introduction to Compact Transformation Groups}, Pure and Applied Mathematics 46, Academic Press (1972)
\bibitem{But}
J.-B. Butruille:
\emph{Classification des varietes approximativement kähleriennes homogenes}, Ann. Global Anal. Geom. 27, no. 3, 201--225 (2005)
\bibitem{RicIt} T.~Buttsworth, M.~Hallgren: \emph{Local stability of Einstein metrics under the Ricci iteration}, J. Funct. Anal. 280, 108801 (2021)
\bibitem{CHI04} H.-D.~Cao, R.~Hamilton, T.~Ilmanen: \emph{Gaussian densities and stability for some Ricci
solitons}, \url{https://arxiv.org/abs/math/0404165} (2004)
\bibitem{CH15} H.-D.~Cao, C.~He: \emph{Linear stability of Perelman's $\nu$-entropy on symmetric spaces of compact type}, J. Reine Angew. Math. 709, 229--246 (2015).
\bibitem{CLW} X.~Chen, C.~LeBrun, B.~Weber: \emph{On conformally Kähler, Einstein manifolds}, J. Amer. Math. Soc. 21, 1137--1168 (2008)
\bibitem{CDS} X.~Chen, S.~Donaldson, S.~Sun: \emph{Kähler--Einstein metrics on Fano manifolds}, J. Amer. Math. Soc. 28, I: 183--197, II: 199--234, III: 235--278 (2014)
\bibitem{Coe12} C.~van Coevering: \emph{Sasaki–Einstein 5-manifolds associated to toric 3-Sasaki manifolds}, New York J. Math. 18, 555--608 (2012).
\bibitem{Coe17} C.~van Coevering: \emph{Deformations of Killing spinors
on Sasakian and 3-Sasakian manifolds}, J. Math. Soc. Japan 69 (1), 53--91 (2017)
\bibitem{CS19} T.~C.~Collins, G.~Székelyhidi: \emph{Sasaki--Einstein metrics and K-stability}, Geom. Topol. 23 (3), 1339--1413 (2019)
\bibitem{DK24} M.~Dahl, K.~Kröncke: \emph{Local and global scalar curvature rigidity of Einstein manifolds}, Math. Ann. 388, 453--510 (2024)
\bibitem{DWW05}  X.~Dai, X.~Wang, G.~Wei: \emph{On the stability of Riemannian manifold with parallel spinors}, Invent. math. 161, 151–176 (2005)
\bibitem{Dai07} X.~Dai: \emph{Stability of Einstein Metrics and Spin Structures}, Proceedings of the 4th International Congress of Chinese Mathematicians, Vol. II, 59-72 (2007)
\bibitem{DWW07} X.~Dai, X.~Wang, G.~Wei: \emph{On the variational stability of Kähler–Einstein metrics}, Comm. Anal. Geom. 14 (4), 669-–693 (2007)
\bibitem{DZ79} J.~E.~D'Atri, W.~Ziller: \emph{Naturally reductive metrics and Einstein metrics on compact Lie groups}, Memoirs AMS 18:215 (1979)
\bibitem{D}
E.~Delay:
\emph{TT--eigentensors for the Lichnerowicz Laplacian on some asymptotically hyperbolic manifolds with warped products metrics},
Manuscripta Math. 123, no. 2, 147--165 (2007)
\bibitem{FIN05} M.~Feldman, T.~Ilmanen, L.~Ni: \emph{Entropy and reduced distance for Ricci expanders}, J. Geom. Anal. 15, no. 1, 49--62 (2005)
\bibitem{FKS21} J.~Fine, K.~Krasnov, M.~Singer: \emph{Local rigidity of Einstein 4-manifolds satisfying a chiral curvature condition}, Math. Ann. 379, 569--588 (2021)
\bibitem{FM73} A.~E.~Fischer, J.~E.~Marsden: \emph{Linearization stability of the Einstein equations}, Bull. Amer. Math. Soc. 79:5, 997--1003 (1973)
\bibitem{FW} A.~Fischer, J.~Wolf:
\emph{The structure of compact Ricci-flat Riemannian manifolds}, J. Diff. Geom. 10, 277--288 (1975)
\bibitem{F17}
L. Foscolo:
\emph{Deformation theory of nearly Kähler manifolds},
J. Lond. Math. Soc. (2) 95, no. 2, 586--612 (2017)
\bibitem{FH}
L. Foscolo, M. Haskins:
\emph{New $G_2$-holonomy cones and exotic nearly Kähler structures on $S^6$ and $S^3\times S^3$},
Ann. of Math. (2) 185, no. 1, 59--130 (2017)
\bibitem{FKMS} T. Friedrich, I. Kath, A. Moroianu, U. Semmelmann:
\emph{On nearly parallel $G_2$-structures},
J. Geom. Phys. 23, no. 3--4, 259--286  (1997)
\bibitem{Fuj79} T.~Fujitani: \emph{Compact suitably pinched Einstein manifolds}, Bull. Fac. Liberal Arts Nagasaki Univ. 19, 1--5 (1979)
\bibitem{GG96} J.~Gasqui, H.~Goldschmidt: \emph{Radon transforms and spectral rigidity on the complex quadrics and the real Grassmannians of rank two}, J. Reine Angew. Math. 480, 1--69 (1996)
\bibitem{GH02} G.~W.~Gibbons, S.~A.~Hartnoll: \emph{Gravitational instability in higher dimensions}, Phys. Rev. D 66, 064024 (2002)
\bibitem{GHP03} G.~W.~Gibbons, S.~A.~Hartnoll, C.~N.~Pope: \emph{Bohm and Einstein--Sasaki metrics, black holes, and cosmological event horizons}, Phys. Rev. D (3) 67, no. 8, 084024 (2003)
\bibitem{Goto} R.~Goto: \emph{Moduli spaces of topological calibrations, Calabi--Yau, hyperkähler, $\rmG_2$ and $\Spin(7)$ structures}, Internat. J. Math. 115 (3), 211–257 (2004)
\bibitem{GL23} V.~Gutiérrez, J.~Lauret: \emph{Stability of standard Einstein metrics on homogeneous spaces of non-simple Lie groups}, Collectanea Math. 76, 273--284 (2025)
\bibitem{G24} V.~Gutiérrez:
\emph{Stability of non-diagonal Einstein metrics on homogeneous spaces $H\times H/\Delta K$}, Diff. Geom. Appl. 101, 102295 (2025)
\bibitem{HHS}
S.~Hall, R.~Haslhofer, M.~Siepmann:
\emph{The stability inequality for Ricci--flat cones},
J. Geom. Anal. 24, no. 1, 472--494 (2014)
\bibitem{HMW21} S.~J.~Hall, T.~Murphy, J.~Waldron: \emph{Compact Hermitian Symmetric Spaces, Coadjoint Orbits, and the Dynamical Stability of the Ricci Flow}, J. Geom. Anal. 31, 6195--6218 (2021)
\bibitem{HSS25a} S.~J.~Hall, P.~Schwahn, U.~Semmelmann: \emph{On the rigidity of the complex Grassmannians}, Trans. AMS 378 (6), 4335--4367 (2025)
\bibitem{HSS25b} S.~J.~Hall, P.~Schwahn, U.~Semmelmann: \emph{Sandwich operators and Einstein deformations of compact symmetric spaces related to Jordan algebras}, \url{https://arxiv.org/abs/2412.08770}
\bibitem{Heber} J.~Heber: \emph{Noncompact homogeneous Einstein spaces}, Invent. math. 133, 279--352 (1998)
\bibitem{HMS15}
K. Heil, A. Moroianu, U. Semmelmann:
\emph{Killing and conformal Killing tensors},
J. Geom. Phys. 106, 383--400 (2016)
\bibitem{H06}
Y. Homma:
\emph{Estimating the eigenvalues on quaternionic Kähler manifolds}, Internat. J. Math. 17, no. 6, 665--691 (2006)
\bibitem{HS25} Y.~Homma, U.~Semmelmann: \emph{Eigenvalue estimates and stability on positive quaternion Kähler manifolds}, \url{https://arxiv.org/abs/2604.01333} (2026)
\bibitem{Hori} E.~Horikawa: \emph{Algebraic surfaces of general type with small $c_1^2$, I}, Ann. Math. 104 (2), 357--387 (1976)
\bibitem{IN04} M.~Itoh, T.~Nakagawa: \emph{Variational stability and local rigidity of Einstein metrics}, Yokohama Math. J. 51, 103--115 (2005)
\bibitem{Jen71} G.~Jensen: \emph{The scalar curvature of left-invariant Riemannian metrics}, Indiana U. Math. J. 20, 1125--1144 (1971)
\bibitem{Joyce} D.~D.~Joyce: \emph{Compact Manifolds with Special Holonomy}, Oxford University Press (2000)
\bibitem{Kraemer} M. Krämer, \emph{Eine Klassifikation bestimmter Untergruppen kompakter zusammenhängender Liegruppen}, Comm. Algebra 3 (8), 691--737 (1975)
\bibitem{KS19} D.~Knopf, N.~Šešum: \emph{Dynamic Instability of $\C\Proj^N$ Under Ricci Flow}, J. Geom. Anal. 29, 902--916 (2019)
\bibitem{KNS} K.~Kodaira, L.~Nirenberg, D.~C.~Spencer: \emph{On the existence of deformations of complex analytic structures}, Ann. Math. 68, 450--459 (1958)
\bibitem{Koiso78} N.~Koiso: \emph{Non-deformability of Einstein metrics}, Osaka J.~Math.~15, 419--433 (1978)
\bibitem{Koiso80} N.~Koiso: \emph{Rigidity and stability of Einstein metrics -- The case of compact symmetric spaces}, Osaka J.~Math.~17 (1), 51--73 (1980)
\bibitem{Koiso82} N.~Koiso: \emph{Rigidity and infinitesimal deformability of Einstein metrics}, Osaka J.~Math.~19 (3), 643--668 (1982)
\bibitem{Koiso83} N.~Koiso: \emph{Einstein metrics and complex structures}, Invent. Math.~73 (1), 71--106 (1983).
\bibitem{Kr15a} K.~Kröncke: \emph{On the stability of Einstein manifolds}, Ann. Glob. Anal. Geom. 47, 81–-98 (2015)
\bibitem{Kr15b} K.~Kröncke: \emph{On infinitesimal Einstein deformations}, Diff. Geom. Appl. 38, 41--57 (2015)
\bibitem{Kr15c} K.~Kröncke: \emph{Stability and instability of Ricci solitons}, Calc. Var. 53, 265--287 (2015)
\bibitem{Kr16} K.~Kröncke: \emph{Rigidity and infinitesimal deformability of Ricci solitons}, J. Geom. Anal. 26,
1795--1807 (2016)
\bibitem{Kr17} K.~Kröncke: \emph{Stable and unstable Einstein warped products}, Trans. AMS 369 (9), 6537--6563 (2017)
\bibitem{Kr20} K.~Kröncke: \emph{Stability of Einstein metrics under Ricci flow}, Comm. Anal. Geom. 28 (2), 351--394 (2020)
\bibitem{Kr21} K.~Kröncke: \emph{Spectra, rigidity and stability of sine-cones}, Journal of Functional Analysis 281, 109115 (2021)
\bibitem{KS24} K.~Kröncke, U.~Semmelmann: \emph{On stability and scalar curvature rigidity of quaternion-Kähler manifolds}, to appear in Ann. Sc. Norm. Super. Pisa, Cl. Sci., \url{https://doi.org/10.2422/2036-2145.202501_021} (2025)
\bibitem{Kr25} K.~Kröncke: \emph{Ricci flow and the scalar curvature rigidity of Einstein manifolds}, to appear in the same volume
\bibitem{LL23} E.~A.~Lauret, J.~Lauret: \emph{The stability of standard homogeneous Einstein manifolds}, Mathematische Zeitschrift 303, no. 16 (2023)
\bibitem{LT25} E.~A.~Lauret, A.~Tolcachier: \emph{Linear stability of Perelman's $\nu$-entropy of standard Einstein manifolds}, \url{https://arxiv.org/abs/2506.12435} (2025)
\bibitem{LW22} J.~Lauret, C.~Will: \emph{Prescribing Ricci curvature on homogeneous spaces}, J. Reine Angew. Math 783, 95--133 (2022)
\bibitem{Gstab1} J.~Lauret: \emph{On the stability of homogeneous Einstein manifolds}, Asian J. Math. 26:4, 555--584 (2022)
\bibitem{Gstab2} J.~Lauret, C.~Will: \emph{On the stability of homogeneous Einstein manifolds II}, J. Lond. Math. Soc. 106, 3638--3669 (2022)
\bibitem{Gstab3} J.~Lauret, C.~Will: \emph{Homogeneous Einstein metrics and local maxima of the Hilbert action}, J. Geom. Phys. 178, 104544 (2022)
\bibitem{LW25} J.~Lauret, C.~Will: \emph{Einstein metrics on homogeneous spaces $H\times H/\Delta K$}, Comm. Contemp. Math. (in press) no.~2550010 (2025)
\bibitem{LM} H.~B.~Lawson jr., M.-L.~Michelsohn: \emph{Spin Geometry}, Princeton University Press (1989)
\bibitem{LeB88} C.~LeBrun: \emph{A rigidity theorem for quaternionic-Kähler manifolds}, Proc. Amer. Math. Soc. 103 (4), 1205--1208  (1988)
\bibitem{LeB95} C.~LeBrun: \emph{Einstein metrics and Mostow rigidity}, Math. Res. Lett. 2, 1--8 (1995)
\bibitem{LeB99} C.~LeBrun: \emph{Einstein metrics and the Yamabe problem}, in: Trends in Mathematical Physics, Studies in Adv. Math. 13, AMS/IP (1999)
\bibitem{LeBrunAGAG} C.~LeBrun: \emph{Einstein metrics, harmonic forms, and symplectic four-manifolds}, Ann. Glob. Anal. Geom. 48, 75--85 (2015)
\bibitem{LZ23} Y.~Li, W.~Zhang: \emph{Rigidity of complex projective spaces in Ricci shrinkers}, Calc. Var. 62:171 (2023)
\bibitem{LST24} Y.~Liu, T.~Sano, L.~Tasin: \emph{Infinitely many families of Sasaki--Einstein metrics on spheres}, J. Diff. Geom. 130 (1), 1--26 (2025)
\bibitem{Man1} O. V. Manturov: \emph{Homogeneous, non-symmetric Riemannian spaces with an irreducible rotation group}, Dokl. Akad. Nauk SSSR 141, 792--795 (1961)
\bibitem{Man2} O. V. Manturov: \emph{Riemannian spaces with orthogonal and symplectic motion groups and an irreducible rotation group}, Dokl. Akad. Nauk SSSR 141, 1034--1037  (1961)
\bibitem{Man3} O. V. Manturov: \emph{Homogeneous Riemannian manifolds with irreducible isotropy group}, Trudy Sem. Vector. Tenzor. Anal. 13, 68--145 (1966)
\bibitem{Mon75} V.~Moncrief: \emph{Spacetime symmetries and linearization stability of the Einstein equations. I}, J. Math. Phys. 16, 493--498 (1975)
\bibitem{Mon76} V.~Moncrief: \emph{Spacetime symmetries and linearization stability of the Einstein equations. II}, J. Math. Phys. 17, 1893--1902 (1976)
\bibitem{spinc} A.~Moroianu: \emph{Parallel and Killing Spinors on Spin$^c$ Manifolds}, Commun. Math. Phys. 187, 417--427 (1997)
\bibitem{MS10} A.~Moroianu, U.~Semmelmann: \emph{The Hermitian Laplace Operator on Nearly Kähler Manifolds}, Commun. Math. Phys. 294, 251--272 (2010)
\bibitem{MS11} A.~Moroianu, U.~Semmelmann: \emph{Infinitesimal Einstein deformations of nearly Kähler metrics}, Trans. Amer. Math. Soc. 363, no. 6, 3057--3069 (2011)
\bibitem{NS21} P.-A. Nagy, U. Semmelmann: \emph{Deformations of nearly $G_2$ structures},
J. Lond. Math. Soc. (2) 104, no. 4, 1795--1811  (2021)
\bibitem{NS23} P.-A. Nagy, U. Semmelmann: \emph{Eigenvalue estimates for $3$-Sasaki structures}, J. Reine Angew. Math. 803, 35--60 (2023)
\bibitem{NS24} P.-A. Nagy, U. Semmelmann: \emph{The $G_2$ geometry of $3$-Sasaki structures},  J. Geom. Anal. 34, no. 2, Paper No. 61 (2024)
\bibitem{NS25}  P.-A. Nagy, U. Semmelmann: \emph{Second order Einstein deformations}, J. Math. Soc. Japan 77 (2), (2025)
\bibitem{N00} Yu.~G.~Nikonorov: \emph{Algebraic structure of standard homogeneous Einstein manifolds}, Siberian
Adv. Math. 10, 59--82 (2000)
\bibitem{Nik21} Yu.~G.~Nikonorov: \emph{On a characterization of critical points of the scalar curvature functional} (Russian), Tr. Rubtsovsk. Ind. Inst. 7, 211--217 (2000); English version: \url{https://arxiv.org/abs/2112.00993}
\bibitem{Nor13} J.~Nordström: \emph{Ricci-flat deformations of metrics with exceptional holonomy}, Bull. London Math. Soc. 45, 1004--1018 (2013)
\bibitem{Ozu21} T.~Ozuch: \emph{Integrability of Einstein deformations and desingularizations}, Comm. Pure Appl. Math. 77 (1), 177--220 (2023)
\bibitem{Per02} G.~Perelman: \emph{The entropy formula for the Ricci flow and its geometric applications}, \url{https://arxiv.org/abs/math/0211159} (2002)
\bibitem{petean} J.~Petean: \emph{The Yamabe invariant of simply connected manifolds}, J. reine und angew. Math. 523, 225--231 (2000)
\bibitem{PP}
H. Pedersen, Y.-S. Poon:
\emph{A note on rigidity of 3-Sasakian manifolds}, Proc. Amer. Math. Soc. 127 (1999), no. 10, 3027–3034
\bibitem{PS15} F.~Podestà, A.~Spiro: \emph{On moduli spaces of Ricci solitons}, J. Geom. Anal. 25 (2), 1157–-1174 (2015)
\bibitem{Sch22a} P.~Schwahn: \emph{Stability of Einstein metrics on symmetric spaces of compact type}, Ann. Glob. Anal. Geom. 61, 333--357 (2022)
\bibitem{Sch22b} P.~Schwahn: \emph{Coindex and Rigidity of Einstein Metrics on Homogeneous Gray Manifolds}, J. Geom. Anal. 32, 302 (2022)
\bibitem{SSW23} P.~Schwahn, U. Semmelmann, G. Weingart: \emph{Stability of the non-symmetric space $\rmE_7/\PSO(8)$},  Adv. Math. 432, Paper No. 109268 (2023)
\bibitem{Sch24} P.~Schwahn: \emph{The Lichnerowicz Laplacian on normal homogeneous spaces}, J. Reine Angew. Math. 814, 91--11 (2024)
\bibitem{SWW20} U. Semmelmann, C. Wang, M. Y.-K. Wang:  \emph{On the linear stability of nearly Kähler $6$-manifolds},
Ann. Global Anal. Geom. 57, no. 1, 15-22, (2020)
\bibitem{SWW22} U. Semmelmann, C. Wang, M. Y.-K. Wang:  \emph{Linear instability of Sasaki Einstein and nearly parallel $G_2$ manifolds},
Internat. J. Math. 33, no. 6, Paper No. 2250042, 17 pp.,  (2022)
\bibitem{SW19} U. Semmelmann, G. Weingart: \emph{The standard Laplace operator}, manuscripta math. 158, 273--293 (2018)
\bibitem{SW22} U. Semmelmann, G. Weingart:  \emph{Stability of compact symmetric spaces}, J. Geom. Anal. 32, no. 4, Paper No. 137, 27 pp., (2022)
\bibitem{cohom1coindex} E.~Solé-Farré: \emph{The Hitchin index in cohomogeneity one nearly Kähler structures}, \url{https://arxiv.org/abs/2410.21106} (2024)
\bibitem{Stolz} S.~Stolz: \emph{Positive Scalar Curvature -- Constructions and Obstructions}, in: M.~L.~Gromov, H.~B.~Lawson (eds.), \emph{Perspectives in Scalar Curvature} Vol.~II, 5--49, World Scientific (2023)
\bibitem{Tian90} G.~Tian: \emph{On Calabi's conjecture for complex surfaces with positive first Chern class}, Invent. Math. 101 (1), 101--172 (1990)
\bibitem{Tian15} G.~Tian: \emph{K-Stability and Kähler-Einstein Metrics}, Comm. Pure Appl. Math. 68 (7), 1085--1156 (2015)
\bibitem{W} C.~Wang: \emph{Stability of Riemannian manifolds with Killing spinors}, Internat. J. Math. 28, no. 1, 1750005, 19 pp. (2017)
\bibitem{WW} C.~Wang, M.~Y.~Wang:
\emph{Stability of Einstein metrics on fiber bundles},
J. Geom. Anal. 31, no. 1, 490--515 (2021)
\bibitem{WW2} C.~Wang, M.~Y.~Wang:
\emph{Instability of some Riemannian manifolds with real Killing spinors}, Comm. Anal. Geom. 30, no. 8, 1895--1931 (2022)
\bibitem{WZ85} M.~Y.~Wang, W.~Ziller: \emph{On normal homogeneous Einstein manifolds}, Ann. Sci. École Norm. Sup. 18, 563--633 (1985)
\bibitem{WZ86} M.~Y.~Wang, W.~Ziller: \emph{Existence and nonexistence of homogeneous Einstein metrics}, Invent. math. 84, 177--194 (1986)
\bibitem{WZ90} M.~Y.~Wang, W.~Ziller: \emph{Einstein metrics on principal torus bundles}, J. Diff. Geom. 31, 215--248 (1990)
\bibitem{Wang91} M.~Y.~Wang: \emph{Preserving Parallel Spinors under Metric Deformations}, Indiana Univ. Math. J. 40 (3), 815--844 (1991)
\bibitem{Wolf} J.~A.~Wolf: \emph{The geometry and structure of isotropy irreducible homogeneous spaces}, Acta Math. 120, 59--148 (1968)
\end{thebibliography}
\end{document}